\documentclass{gAPA2e}
\usepackage{float}
\usepackage{graphicx,psfrag,epsf}
\usepackage{epstopdf}
\usepackage{graphicx}
\usepackage{subfigure}
\usepackage{algorithm}
\usepackage{multirow}

\newtheorem{definition}{Definition}
\newtheorem{example}{Example}
\newtheorem{remark}{Remark}
\newtheorem{proposition}{Proposition}
\newtheorem{theorem}{Theorem}
\newtheorem{lemma}{Lemma}

\begin{document}


\title{\itshape On the second order asymptotical regularization of linear ill-posed inverse problems}

\author{Y. Zhang$^{\rm a,b}$$^{\ast}$\thanks{$^\ast$Corresponding author. Email: ye.zhang@mathematik.tu-chemnitz.de.
\vspace{6pt}}, B. Hofmann$^a$ \\\vspace{6pt}  $^{a}${\em{Faculty of Mathematics, Chemnitz University of Technology, Chemnitz 09107, Germany}}; $^{b}${\em{School of Science and Technology, \"{O}rebro University,  \"{O}rebro 70182, Sweden}} }

\maketitle

\begin{abstract}
In this paper, we establish an initial theory regarding the Second Order Asymptotical Regularization (SOAR) method for the stable approximate solution of ill-posed linear operator equations in Hilbert spaces, which are models for linear inverse problems with applications in the natural sciences, imaging and engineering. We show the regularizing properties of the new method, as well as the corresponding convergence rates. We prove that, under the appropriate source conditions and by using Morozov's conventional discrepancy principle, SOAR exhibits the same power-type convergence rate as the classical version of asymptotical regularization (Showalter's method). Moreover, we propose a new total energy discrepancy principle for choosing the terminating time of the dynamical solution from SOAR, which corresponds to the unique root of a monotonically non-increasing function and allows us to also show an order optimal convergence rate for SOAR. A damped symplectic iterative regularizing algorithm is developed for the realization of SOAR. Several numerical examples are given to show the accuracy and the acceleration affect of the proposed method. A comparison with other state-of-the-art methods are provided as well.


\begin{classcode}
47A52;65J20;65F22;65R30
\end{classcode}

\end{abstract}

\section{Introduction}
\label{Introduction}

We are interested in solving linear operator equations,
\begin{eqnarray}\label{OperatorEq}
A x = y,
\end{eqnarray}
where $A$ is an injective and compact linear operator acting between two infinite dimensional Hilbert spaces $\mathcal{X}$ and $\mathcal{Y}$. For simplicity, we denote by $\langle \cdot, \cdot \rangle$ and $\|\cdot\|$ the inner products and norms, respectively, for both Hilbert spaces. Since $A$ is injective, the operator equation (\ref{OperatorEq}) has a unique solution $x^\dagger \in \mathcal{X}$ for every $y$ from range $\mathcal{R}(A)$ of the linear operator $A$. In this context, $\mathcal{R}(A)$ is assumed to be an infinite dimensional subspace of $\mathcal{Y}$.

Suppose that, instead of the exact right-hand side $y=A x^\dagger$, we are given noisy data $y^\delta\in \mathcal{Y}$ obeying the deterministic noise model $\|y^\delta - y\|\leq \delta$ with noise level $\delta>0$. Since $A$ is compact and ${\rm dim}(\mathcal{R}(A))=\infty$, we have $\mathcal{R}(A)\neq \overline{\mathcal{R}(A)}$ and the problem (\ref{OperatorEq}) is ill-posed. Therefore, regularization methods should be employed for obtaining stable approximate solutions.

Loosely speaking, two groups of regularization methods exist: variational regularization methods and iterative regularization methods. Tikhonov regularization is certainly the most prominent variational regularization method (cf., e.g.~\cite{Tikhonov-1998}), while the Landweber iteration is the most famous iterative regularization approach (cf., e.g.~\cite{engl1996regularization,Kaltenbacher-2008}). In this paper, our focus is on the latter, since from a computational viewpoint the iterative approach seems more attractive, especially for large-scale problems.

For the linear problem (\ref{OperatorEq}), the Landweber iteration is defined by
\begin{eqnarray}\label{Linear}
x_{k+1} = x_{k} + \Delta t A^* ( y^\delta- A x_{k} ), \quad \Delta t\in(0, 2/\|A\|^2),
\end{eqnarray}
where $A^*$ denotes the adjoint operator of $A$. We refer to~\cite[\S~6.1]{engl1996regularization} for the regularization property of the Landweber iteration. The continuous analogue of (\ref{Linear}) can be considered as a first order evolution equation in Hilbert spaces
\begin{eqnarray}\label{FisrtFlow}
\dot{x}(t) + A^* A x(t)= A^* y^\delta
\end{eqnarray}
if an artificial scalar time $t$ is introduced, and $\Delta t\to0$ in (\ref{Linear}). Here and later on, we use Newton's conventions for the time derivatives. The formulation (\ref{FisrtFlow}) is known as Showalter's method, or asymptotic regularization~\cite{Tautenhahn-1994,Vainikko1986}. The regularization property of (\ref{FisrtFlow}) can be analyzed through a proper choice of the terminating time. Moreover, it has been shown that by using Runge-Kutta integrators, all of the properties of asymptotic regularization (\ref{FisrtFlow}) carry over to its numerical realization~\cite{Rieder-2005}.

From a computational view point, the Landweber iteration, as well as the steepest descent method and the minimal error method, is quite slow. Therefore, in practice accelerating strategies are usually used; see~\cite{Kaltenbacher-2008,Neubauer-2000} and references therein for details.

Over the last few decades, besides the first order iterative methods, there has been increasing evidence to show that the discrete second order iterative methods also enjoy remarkable acceleration properties for ill-posed inverse problems. The well-known methods are the Levenberg-Marquardt method~\cite{JinQ-2010}, the iteratively regularized Gauss-Newton method~\cite{JinQ-2009}, the $\nu$-method~\cite[\S~6.3]{engl1996regularization}, and the Nesterov acceleration scheme~\cite{Neubauer-2017}. Recently, a more general second order iterative method~-- the two-point gradient method~-- has been developed, in~\cite{Hubmer-2017}. In order to understand better the intrinsic properties of the discrete second order iterative regularization, we consider in this paper the continuous version
\begin{eqnarray}\label{SecondFlow}
\left\{\begin{array}{ll}
\ddot{x}(t) + \eta \dot{x}(t) + A^* A x(t)= A^* y^\delta,  \\
x(0)=x_0, \quad \dot{x}(0)= \dot{x}_0.
\end{array}\right.
\end{eqnarray}
of the second order iterative method in the form of an evolution equation,
where $x_0 \in \mathcal{X}$ and $\dot{x}_0 \in \mathcal{X}$ are the prescribed initial data and $\eta>0$ is a constant damping parameter.

From a physical viewpoint, the system (\ref{SecondFlow}) describes the motion of a heavy ball that rolls over the graph of the residual norm square functional $\Phi(x)=\|y^\delta- A x\|^2$ and that keeps rolling under its own inertia until friction stops it at a critical point of $\Phi(x)$. This nonlinear oscillator with damping, which is called the \emph{Heavy Ball with Friction} (HBF) system, has been considered by several authors from an optimization viewpoint, establishing different convergence results and identifying the circumstances under which the rate of convergence of HBF is better than the one of the first order methods; see~\cite{Alvarez-2000,Alvarez-2002,Attouch-2000}. Numerical algorithms based on (\ref{SecondFlow}) for solving some special problems, e.g. inverse source problems in partial differential equations, large systems of linear equations, and the nonlinear Schr\"{o}dinger problem, etc., can be found in~\cite{ZhangYe2018,Edvardsson-2015,Edvardsson-2012,Sandin-2016}. The main goal of this paper is the intrinsic structure analysis of the theory of the second order iterative regularization and the development of new iterative regularization methods based on the framework (\ref{SecondFlow}).

The remainder of the paper is structured as follows: in Section 2, we extend the theory of general affine regularization schemes for solving linear ill-posed problems to a more general setting, adapted for
the analysis of the second order model (\ref{SecondFlow}). Then, the existence and uniqueness of the second order flow (\ref{SecondFlow}), as well as some of its properties, are discussed in Section 3. Section 4 is devoted to the study of the regularization property of the dynamical solution to (\ref{SecondFlow}), while Section 5 presents the results about convergence rates under the assumption of conventional source conditions. In Section 6, based on the St\"{o}rmer-Verlet method, we develop a novel iterative regularization method for the numerical implementation of the second order asymptotical regularization. Some numerical examples, as well as a comparison with four existing iterative regularization methods, are presented in Section 7. Finally, concluding remarks are given in Section 8.


\section{General affine regularization methods}
In this section, we consider general affine regularization schemes based on a family of pairs of piecewise continuous functions $\{g_\alpha (\lambda), \phi_\alpha (\lambda) \}_{\alpha}$ ($0<\lambda\leq \|A\|^2$) for regularization parameters $\alpha\in(0,\bar{\alpha}]$. Once a pair of generating functions $\{g_\alpha (\lambda), \phi_\alpha (\lambda) \}$ is chosen, the approximate solution to (\ref{OperatorEq}) can be given by the procedure
\begin{eqnarray}\label{regularization}
x^\delta_\alpha = (1- A^* A g_\alpha(A^* A)) x_0 + \phi_\alpha(A^* A)\dot{x}_0 + g_\alpha(A^* A)A^* y^\delta.
\end{eqnarray}

\begin{remark}
The affine regularization procedure defined by formula (\ref{regularization}) is designed in particular for the second order evolution equation  (\ref{SecondFlow}). If one sets $(x_0, \dot{x}_0)=(0,0)$, the proposed regularization method coincides with the classical linear regularization schema for general linear ill-posed problems; see, e.g.,~\cite{Hofmann-2007}. However, as the numerical experiments in Section 7 will show, the initial data influence the behaviour of the regularized solutions obtained by (\ref{SecondFlow}). By finding an appropriate choice of the triple $(x_0,\dot{x}_0,\eta)$, the second order analog of the asymptotical regularization yields an accelerated procedure with approximate solutions of higher accuracy.
\end{remark}

To evaluate the regularization error $e(x^\dagger, \alpha, \delta) : = \|x^\delta_\alpha - x^\dagger\|$ for the procedure (\ref{regularization}) in combination with the the noise-free intermediate quantity
\begin{eqnarray}\label{ApproximateNoisyFreeSolution}
x_\alpha:= (1- A^* A g_\alpha(A^* A)) x_0 + \phi_\alpha(A^* A)\dot{x}_0 + g_\alpha(A^* A)A^* A x^\dagger,
\end{eqnarray}
where evidently $e(x^\dagger, \alpha, \delta)\le \|x_\alpha-x^\dagger\|+\|x_\alpha^\delta-x_\alpha\|$,
we introduce the concepts of index functions and profile functions from \cite{Mathe-2003} and \cite{Hofmann-2007}, as follows:

\begin{definition}
\label{Index}
A real function $\varphi: (0,\infty) \to (0,\infty)$ is called an index function if it is continuous, strictly increasing and satisfies the condition $\lim_{\lambda\to 0+} \varphi(\lambda)=0$.
\end{definition}

\begin{definition}
\label{Profile}
An index function $f$, for which $\|x_\alpha-x^\dagger\| \le f(\alpha)\;(\alpha\in(0,\bar{\alpha}])$ holds, is called a profile function to $x^\dagger$ under the assumptions stated above.
\end{definition}

Having a profile function $f$ estimating the noise-free error $\|x_\alpha-x^\dagger\|$ and taking into account that $\delta \|g_\alpha(A^* A)A^*\|$ is an upper bound for the noise-propagation error $\|x_\alpha^\delta-x_\alpha\|$,
which is independent of $x^\dagger$, we can estimate the total regularization error $e(x^\dagger, \alpha, \delta)$ as
\begin{eqnarray}\label{errorBounds}
e(x^\dagger, \alpha, \delta) \le  f(\alpha) + \delta \|g_\alpha(A^* A)A^*\|
\end{eqnarray}
for all $\alpha\in(0,\bar{\alpha}]$.
If we denote by
\begin{eqnarray}\label{residualFun}
r_\alpha(\lambda) = 1 - \lambda g_\alpha(\lambda), \quad \lambda\in(0, \|A\|^2]
\end{eqnarray}
the bias function related to the major part of the regularization method $g_\alpha$ from (\ref{regularization}), then $f(\alpha)= \|r_\alpha(A^* A) (x_0 - x^\dagger) + \phi_\alpha(A^* A)\dot{x}_0\|$ is evidently a profile function to $x^\dagger$ and we have
\begin{eqnarray}\label{errorBounds2}
e(x^\dagger, \alpha, \delta) \le  \|r_\alpha(A^* A) (x_0 - x^\dagger) + \phi_\alpha(A^* A)\dot{x}_0\|+\delta \|g_\alpha(A^* A)A^*\|.
\end{eqnarray}

\begin{proposition}
\label{gAlpha}
Assume that the pairs of functions $\{g_\alpha (\lambda), \phi_\alpha (\lambda) \}_{\alpha}$ are piecewise continuous in $\alpha$ and satisfy the following three conditions:
\begin{itemize}
\item[(i)] For any fixed $\lambda\in (0,\|A\|^2]$: $\lim_{\alpha\to0} r_{\alpha}(\lambda) =0$ and $\lim_{\alpha\to0} \phi_{\alpha}(\lambda) =0$.
\item[(ii)] Two constants $\gamma_1$ and $\gamma_2$ exist such that $|r_{\alpha}(\lambda)|\leq \gamma_1$ and $|\phi_{\alpha}(\lambda)|\leq \gamma_2$ hold for all $\alpha\in(0,\bar{\alpha}]$.
\item[(iii)] A constant $\gamma_*$ exists such that $\sqrt{\lambda} |g_{\alpha}(\lambda)| \leq \gamma_*/\sqrt{\alpha}$ for all $\alpha\in(0,\bar{\alpha}]$.
\end{itemize}
Then, if the regularization parameter $\alpha=\alpha(\delta,y^\delta)$ is chosen so that
$$\lim_{\delta\to 0} \alpha =  \lim_{\delta\to 0} \delta^2 / \alpha = 0,$$
the approximate solution in (\ref{regularization}) converges to the exact solution $x^\dagger$ as $\delta\to0$.
\end{proposition}

\begin{proof}
From the properties (i) and (ii) of Proposition \ref{gAlpha} we deduce for $\alpha\to0$ point-wise convergence $r_\alpha(A^* A)x_1\to0$ and $\phi_\alpha(A^* A)x_2\to0$ for any $x_{1,2}\in \mathcal{X}$ (see, e.g.,~\cite[Theorem 4.1]{engl1996regularization}). Therefore, by the estimate (\ref{errorBounds2}) we can derive that
\begin{eqnarray*}
\begin{array}{ll}
e(x^\dagger, \alpha, \delta) \leq \|r_\alpha(A^* A) (x_0 - x^\dagger)\| + \|\phi_\alpha(A^* A)\dot{x}_0\| + \gamma_* \delta / \sqrt{\alpha} \to 0
\end{array}
\end{eqnarray*}
as $\delta\to0$.

\end{proof}

Proposition~\ref{gAlpha} motivates us to call the procedure (\ref{regularization}) a regularization method for the linear inverse problem (\ref{OperatorEq}) if the pair of functions $\{g_\alpha (\lambda), \phi_\alpha (\lambda) \}_{\alpha}$  satisfies the three requirements $(i)$, $(ii)$ and $(iii)$.

\begin{example}
For the Landweber iteration (\ref{Linear}) with the step size $\Delta t\in (0, 2/\|A\|^2)$, we have $\phi_\alpha (\lambda)=0$ and $g_\alpha (\lambda)= (1- (1- \Delta t \lambda)^{\lfloor 1/ \alpha \rfloor})/\lambda$. It is not difficult to show, e.g. in~\cite{Hofmann-2007,Mathe-2006}, that $g_\alpha (\lambda)$ is a regularization method by Proposition \ref{gAlpha} with constants $\gamma_2=0$, $\gamma_1=1$ and $\gamma_*=\sqrt{2\Delta t}$. Consider the continuous version of the Landweber iteration (\ref{FisrtFlow}), i.e. Showalter's method. It is not difficult to show that $\phi_\alpha (\lambda)=0$ and $g_\alpha (\lambda)= (1 - e^{-\lambda/\alpha})/\lambda$, and hence $r_\alpha (\lambda) = e^{-\lambda / \alpha}$. Obviously, $g_\alpha (\lambda)$ is a regularization method with $\gamma_2=0$, $\gamma_1=1$ and $\gamma_*=\theta$ by noting that $\sup_{\lambda\in \mathbf{R}_+} \sqrt{\lambda} g_\alpha (\lambda) = \theta \sqrt{\alpha}$ and $\theta = \sup_{\lambda\in \mathbf{R}_+} \sqrt{\lambda} (\lambda - e^{-\lambda}) \approx 0.6382$~\cite{Vainikko1986}.
\end{example}

Note that the three requirements $(i) - (iii)$ in Proposition \ref{gAlpha} are not enough to ensure rates of convergence for the regularized solutions. More precisely, for rates in the case of ill-posed problems, additional smoothness assumptions on $x^\dagger$ in correspondence with the forward operator $A$ and the regularization method under consideration have to be fulfilled. This allows us to verify the specific profile functions $f(\alpha)$ in formula (\ref{errorBounds}) that are specified for our second order method in formula (\ref{errorBounds2}).
Once a profile function $f$ is given, together with the property $(iii)$ in Proposition~\ref{gAlpha}, we obtain from the estimate (\ref{errorBounds}) that
\begin{eqnarray}\label{Estimator2}
e(x^\dagger, \alpha, \delta) \leq f(\alpha) + \gamma_* \delta/\sqrt{\alpha} \quad \textrm{~for all~} \alpha\in (0, \bar{\alpha}].
\end{eqnarray}
Moreover, if we consider the auxiliary index function
\begin{eqnarray}\label{Theta}
\Theta(\alpha) := \sqrt{\alpha} f(\alpha),
\end{eqnarray}
and choose the regularization parameter a priori as $\alpha_* = \Theta^{-1}(\delta)$, then we can easily see that
\begin{eqnarray}\label{Estimator3}
e(x^\dagger, \alpha_*, \delta) \leq (1 + \gamma_*) f(\Theta^{-1}(\delta)).
\end{eqnarray}
Hence, the convergence rate $f\left(\Theta^{-1}(\delta) \right)$ of the total regularization error as $\delta \to 0$ depends on the profile function $f$ only, but for our suggested approach, $f$ is a function of $x^\dagger,  g_\alpha, \phi_\alpha, x_0, \dot{x}_0, A$ and on the damping parameter $\eta$.

In order to verify the profile function $f$ in detail, it is of interest to consider how sensitive the regularization method is with respect to a priori smoothness assumptions. In this context, the concept of qualification can be exploited for answering this question: the higher its qualification, the more the method is capable of reacting to smoothness assumptions. Expressing the qualification by means of index functions $\psi$, the traditional concept of qualifications with monomials $\psi(\lambda)=\lambda^\kappa$ for $\kappa> 0$  from \cite{Vainikko1986} (see also \cite{engl1996regularization}) has been extended  in \cite{Hofmann-2007,Mathe-2003}  to general index functions $\psi$. We adapt this generalized concept to our class of methods (\ref{regularization}) in the following definition.

\begin{definition}
\label{QualificationDef}
Let $\psi$ be an index function. A regularization method (\ref{regularization}) for the linear operator equation (\ref{OperatorEq}) generated by the pair $\{g_\alpha (\lambda), \phi_\alpha (\lambda)\}\;(0<\lambda\leq \|A\|^2)$ is said to have the qualification $\psi$ with constant $\gamma> 0$ if both inequalities
\begin{eqnarray}\label{QualificationFun}
\sup_{\lambda\in(0,\|A\|^2]} |r_{\alpha}(\lambda)| \psi(\lambda) \leq \gamma \psi(\alpha) \quad \textrm{~and~} \quad \sup_{\lambda\in(0,\|A\|^2]} |\phi_{\alpha}(\lambda)| \psi(\lambda) \leq \gamma \psi(\alpha)
\end{eqnarray}
are satisfied for all $\,0<\alpha\leq \|A\|^2$.
\end{definition}

\begin{remark}
Since the bias function of Showalter's method equals $r_{\alpha}(\lambda)= e^{-\lambda/\alpha}$ and $\phi_\alpha (\lambda)=0$, set $\xi=\lambda$ in the following identity
\begin{eqnarray}\label{Pflimiting1}
\sup_{0\leq \xi < \infty} e^{-\xi / \alpha} \xi^p =  \left( p/e \right)^p \alpha^{p}
\end{eqnarray}
to conclude that for all exponents $p>0$ the monomials $\psi(\lambda) = \lambda^{p}$  are qualifications for Showalter's method. We will show that an analogue result also holds for the second order asymptotical regularization method - see Proposition~\ref{ThmQualificationCase1} below - and will apply this fact to obtaining associated convergence rates.
\end{remark}

\section{Properties of the second order flow}

We first prove the existence and uniqueness of strong global solutions of the second order equation (\ref{SecondFlow}). Then, we study the long-term behavior of the dynamical solution $x(t)$ of (\ref{SecondFlow}) and the residual norm functional $\| A x(t)-y^\delta\|$.

\begin{definition}
$x:[0, +\infty)\to \mathcal{X}$ is a strong global solution of (\ref{SecondFlow}) with initial data $(x_0,\dot{x}_0)$ if $x(0)= x_0 \in \mathcal{X}, \dot{x}(0)= \dot{x}_0 \in \mathcal{X}$, and
\begin{itemize}
\item $x(\cdot), \dot{x}(\cdot):[0, +\infty)\to \mathcal{X}$ are locally absolutely continuous~\cite{Bot-2016},
\item $\ddot{x}(t) + \eta \,\dot{x}(t) + A^* A x(t)= A^* y^\delta$ holds for almost every $t\in[0,+\infty)$.
\end{itemize}
\end{definition}

\begin{theorem}\label{ExsisteceUnique}
For any pair $(x_0,\dot{x}_0)\in \mathcal{X}\times \mathcal{X}$ there exists a unique strong global solution of the second order dynamical system (\ref{SecondFlow}).
\end{theorem}

\begin{proof}
Denote by $\textbf{z}=(x, \dot{x})^T$, and rewrite (\ref{SecondFlow}) as a first order differential equation
\begin{eqnarray}\label{FisrtOrderForm}
\dot{\textbf{z}}(t) = B \textbf{z}(t) + \mathbf{d},
\end{eqnarray}
where $B=-[0, I; A^* A, \eta I]$, $\mathbf{d} = [0; A^* y^\delta]$ and $I$ denotes the identity operator in $\mathcal{X}$. Since $A$ is a bounded linear operator, both $A^*$ and $B$ are also bounded linear operators. Hence, by the Cauchy-Lipschitz-Picard theorem, the first-order autonomous system (\ref{FisrtOrderForm}) has a unique global solution for the given initial data $(x_0,\dot{x}_0)$.

\end{proof}

Now, we start to investigate the long-term behaviors of the dynamical solution and the residual norm functional. These properties will be used for the study of convergence rate in Section 5.

\begin{lemma}\label{LemmaVanishing}
Let $x(t)$ be the solution of (\ref{SecondFlow}). Then, $\dot{x}(\cdot)\in L^2([0,\infty),\mathcal{X})$ and $\dot{x}(t)\to0$ as $t\to\infty$. Moreover, we have the following two limit relations
\begin{equation}
\label{ResidualErrorLimit}
\lim_{t\to \infty} \|A x(t)-y^\delta\| \leq \delta
\end{equation}
and
\begin{eqnarray}
\label{RateLimits}
\lim_{t\to \infty} \|A x(t)-y^\delta\|^2+ \|\dot{x}(t)\|^2 = \inf_{x^*\in \mathcal{X}} \|A x^*-y^\delta\|^2.
\end{eqnarray}
\end{lemma}

The proof of the above lemma uses the idea given in~\cite{Attouch-2000}, and can be found in the Appendix A.1. If $y^\delta$ does not belong to the domain $\mathcal{D}(A^\dagger)$ of the Moore-Penrose inverse $A^\dagger$ of $A$, it is not difficult to show that there is a ``blow-up'' for the solution $x(t)$ of the dynamical system (\ref{SecondFlow}) in the sense that $\|x(t)\|\to\infty$ as $t\to\infty$. Contrarily, for $y^\delta \in \mathcal{D}(A^\dagger)$, i.e.~if the noisy data $y^\delta$ satisfy the Picard condition, one can show more assertions concerning the long-term behaviour of the solution to the evolution equation (\ref{SecondFlow}), and we refer to Lemma~\ref{LemmaVanishingExact} below for results in the case of noise-free data $y=Ax^\dagger$. In this work, for the inverse problem with noisy data, we are first and foremost interested in the case that $y^\delta \not\in \mathcal{D}(A^\dagger)$ may occur, since the set $\mathcal{D}(A^\dagger)$ is of the first category and the chance to meet such an element is negligible.

At the end of this section, we show some properties of $x(t)$ of (\ref{SecondFlow}) with noise-free data.

\begin{lemma}\label{LemmaVanishingExact}
Let $x(t)$ be the solution of (\ref{SecondFlow}) with the exact right-hand side $y$ as data. Then, in the case $\eta\geq\|A\|$, we have
\begin{itemize}
\item[(i)] $x(\cdot)\in L^\infty([0,\infty),\mathcal{X})$.
\item[(ii)] $\dot{x}(\cdot)\in L^\infty([0,\infty),\mathcal{X}) \cap L^2([0,\infty),\mathcal{X})$ and $\dot{x}(t)\to0$ as $t\to\infty$.
\item[(iii)] $\ddot{x}(\cdot)\in L^\infty([0,\infty),\mathcal{X}) \cap L^2([0,\infty),\mathcal{X})$ and $\ddot{x}(t)\to0$ as $t\to\infty$.
\item[(iv)] $\|A x(t) - y\| = o(t^{-1/2})$ as $t\to\infty$.
\end{itemize}
\end{lemma}

The proof of Lemma~\ref{LemmaVanishingExact} follows as a special case for $f(x)=\frac{1}{2} \|A x - y\|^2$ in~\cite{Bot-2016}, and it is given in the Appendix A.2. The rate $\|A x(t) - y\| = o(t^{-1/2})$ as $t \to \infty$ given in Lemma \ref{LemmaVanishingExact} for the second order evolution equation (\ref{SecondFlow}) should be compared with the corresponding result for the first order method, i.e. the gradient decent methods, where one only obtains  $\|A x(t) - y\| = \mathcal{O}(t^{-1/2})$ as $t \to \infty$. If we consider a discrete iterative method with the number $k$ of iterations, assertion (iv) in Lemma \ref{LemmaVanishingExact} indicates that in comparison with gradient descent methods, the second order methods (\ref{SecondFlow}) need the same computational complexity for the number $k$ of iterations, but can achieve a higher order $o(k^{-1/2})$ of accuracy for the objective functional as $k \to \infty$.


\section{Convergence analysis for noisy data}
\label{Regularization}

This section is devoted to the verification of the pair $\{g_\alpha (\lambda), \phi_\alpha (\lambda) \}_\alpha$ of generator functions occurring in formula (\ref{ApproximateNoisyFreeSolution}) associated with the second order equation problem (\ref{SecondFlow}) with the inexact right-hand side $y^\delta$ and the corresponding regularization properties.

Let $\{\sigma_j; u_j, v_j\}_{j=1}^\infty$ be the well-defined singular system for the compact and injective linear operator $A$, i.e.~we have $Au_j= \sigma_j v_j$ and $A^* v_j = \sigma_j u_j$ with ordered singular values $\|A\|=\sigma_1 \geq \sigma_2 \geq \cdot\cdot\cdot \geq \sigma_j \geq \sigma_{j+1} \geq \cdot\cdot\cdot \to 0$ as $j \to \infty$. Since the eigenelements $\{u_j\}_{j=1}^\infty$ and $\{v_j\}_{j=1}^\infty$ form complete orthonormal systems in  $\mathcal{X}$ and $\mathcal{Y}$, respectively, the equation in (\ref{SecondFlow}) is equivalent to
\begin{eqnarray}\label{SVDEq1}
\langle \ddot{x}(t), u_j \rangle + \eta \langle \dot{x}(t) , u_j \rangle + \sigma^2_j \langle x(t) , u_j \rangle = \sigma_j \langle y^\delta , v_j \rangle, \quad j=1,2, ...\,.
\end{eqnarray}
Using the decomposition $x(t)=\sum_j \xi_j(t) u_j$ under the basis $\{u_j\}_{j=1}^\infty$ in $\mathcal{X}$, we obtain
\begin{eqnarray}\label{SVDEq2}
\ddot{\xi}_j(t) + \eta \dot{\xi}_j(t) + \sigma^2_j \xi_j = \sigma_j \langle y^\delta , v_j \rangle, \quad j=1,2, ...\,.
\end{eqnarray}

In order to solve the above differential equation, we have to distinguish three different cases: (a) the overdamped case: $\eta>2\|A\|$, (b) the underdamped case: there is an index $j_0$ such that $2\sigma_{j_0+1} < \eta< 2\sigma_{j_0}$, and (c) the critical damping case: an index $j_0$ exists such that $\eta=2\sigma_{j_0}$. In this section, we discuss for simplicity the overdamped case only. The other two cases are studied similarly, and the corresponding details can be found in the Appendix B. We remark that all results that concluded in the overdamped case also hold for the other two cases, but with different value of positive constants $\gamma_{1,2}, \gamma_*$ in Proposition~\ref{gAlpha} and $\gamma$ in Definition~\ref{QualificationDef}.

In the overdamped case, the characteristic equation of (\ref{SVDEq2}), possessing the form $\ddot{\xi}_j(t) + \eta \dot{\xi}_j(t) + \sigma^2_j \xi_j =0$, which has two independent solutions $\xi^1_j=e^{-\eta t/2} e^{\omega_j t/2}$ and
\linebreak $\xi^2_j=e^{-\eta t/2} e^{-\omega_j t/2}$ for all $j=1,2,...$, where $\omega_j=\sqrt{\eta^2 - 4 \sigma^2_j}>0$. Hence, it is not difficult to show that the general solution to (\ref{SVDEq2}) in the overdamped case is
\begin{eqnarray}\label{GeneralSolution}
\xi_j(t) = c^1_j e^{-\eta t/2} e^{\omega_j t/2} + c^2_j e^{-\eta t/2} e^{-\omega_j t/2} + \sigma^{-1}_j \langle y^\delta , v_j \rangle, \quad j=1,2, ...\,.
\end{eqnarray}

Introducing the initial conditions in (\ref{SecondFlow}) to obtain a system for $\{c^1_j, c^2_j\}$ yields
\begin{eqnarray}\label{EqC1C2}
\left\{\begin{array}{ll}
\sum_j \left( c^1_j + c^2_j + \sigma^{-1}_j \langle y^\delta , v_j \rangle \right) u_j = x^0 ,  \\
\sum_j \left( \frac{\eta - \omega_j}{2} c^1_j + \frac{\eta + \omega_j}{2} c^2_j \right) u_j = \dot{x}^0,
\end{array}\right.
\end{eqnarray}
or equivalently with the decomposition $x_0=\sum_j \langle x_0 , u_j \rangle u_j$ for all $j=1,2,...$
\begin{eqnarray}\label{EqC1C2j}
\left\{\begin{array}{ll}
c^1_j + c^2_j + \sigma^{-1}_j \langle y^\delta , v_j \rangle = \langle x_0 , u_j \rangle,  \\
\frac{\eta - \omega_j}{2} c^1_j + \frac{\eta + \omega_j}{2} c^2_j = \langle \dot{x}_0 , u_j \rangle,
\end{array}\right.
\end{eqnarray}
which gives
\begin{eqnarray}\label{C1C2}
\left\{\begin{array}{ll}
c^1_j = \frac{\eta + \omega_j}{2\omega_j} \langle x_0 , u_j \rangle - \frac{1}{\omega_j} \langle \dot{x}_0 , u_j \rangle - \frac{\eta + \omega_j}{2\omega_j \sigma_j}  \langle y^\delta , v_j \rangle ,  \\
c^2_j =  - \frac{\eta - \omega_j}{2\omega_j} \langle x_0 , u_j \rangle + \frac{1}{\omega_j} \langle \dot{x}_0 , u_j \rangle + \frac{\eta - \omega_j}{2\omega_j \sigma_j}  \langle y^\delta , v_j \rangle .
\end{array}\right. \qquad j=1,2, ...\,.
\end{eqnarray}
By a combination of (\ref{C1C2}), the definition of $\omega_j$ and the decomposition of $x(t)$ we obtain
\begin{eqnarray*}\label{xDeltaRegu}
\begin{array}{ll}
x(t) = \sum\limits_j \left( \frac{\eta + \sqrt{\eta^2 - 4 \sigma^2_j}}{2\sqrt{\eta^2 - 4 \sigma^2_j}} e^{- \frac{\eta - \sqrt{\eta^2 - 4 \sigma^2_j}}{2} t} - \frac{\eta - \sqrt{\eta^2 - 4 \sigma^2_j}}{2\sqrt{\eta^2 - 4 \sigma^2_j}} e^{-\frac{\eta + \sqrt{\eta^2 - 4 \sigma^2_j}}{2} t} \right) \langle x_0 , u_j \rangle u_j \\ \qquad\qquad
- \sum\limits_j  \frac{1}{2\sqrt{\eta^2 - 4 \sigma^2_j}} \left( e^{- \frac{\eta - \sqrt{\eta^2 - 4 \sigma^2_j}}{2} t} - e^{-\frac{\eta + \sqrt{\eta^2 - 4 \sigma^2_j}}{2} t} \right) \langle \dot{x}_0 , u_j \rangle u_j \\ \qquad
+ \sum\limits_j \frac{1- \left( \frac{\eta + \sqrt{\eta^2 - 4 \sigma^2_j}}{2\sqrt{\eta^2 - 4 \sigma^2_j}} e^{- \frac{\eta - \sqrt{\eta^2 - 4 \sigma^2_j}}{2} t} - \frac{\eta - \sqrt{\eta^2 - 4 \sigma^2_j}}{2\sqrt{\eta^2 - 4 \sigma^2_j}} e^{-\frac{\eta + \sqrt{\eta^2 - 4 \sigma^2_j}}{2} t} \right) }{\sigma_j } \langle y^\delta , v_j \rangle u_j  \\ \qquad
=: (1- A^* A g(t, A^* A)) x_0 + \phi(t, A^* A)\dot{x}_0 + g(t, A^* A) A^* y^\delta,
\end{array}
\end{eqnarray*}
where
\begin{eqnarray}\label{gPhiDef}
\left\{\begin{array}{ll}
g(t, \lambda) = \frac{1}{\lambda} \left( 1-  \frac{\eta + \sqrt{\eta^2 - 4 \lambda}}{2\sqrt{\eta^2 - 4 \lambda}} e^{- \frac{\eta - \sqrt{\eta^2 - 4 \lambda}}{2} t} + \frac{\eta - \sqrt{\eta^2 - 4 \lambda}}{2\sqrt{\eta^2 - 4 \lambda}} e^{-\frac{\eta + \sqrt{\eta^2 - 4 \lambda}}{2} t} \right) , \\
\phi(t, \lambda) = - \frac{1}{2\sqrt{\eta^2 - 4 \lambda}} \left( e^{- \frac{\eta - \sqrt{\eta^2 - 4 \lambda}}{2} t} - e^{-\frac{\eta + \sqrt{\eta^2 - 4 \lambda}}{2} t} \right) .
\end{array}\right.
\end{eqnarray}
We find the form required for the generator functions in formula (\ref{ApproximateNoisyFreeSolution}) if we set
\begin{equation}\label{gAlphaFun}
g_\alpha(\lambda) := g(1/\alpha, \lambda)\quad  \textrm{~and~} \quad \phi_\alpha(\lambda) := \phi(1/\alpha, \lambda).
\end{equation}
Then the corresponding bias function $r_\alpha(\lambda) = 1 - \lambda g(1/\alpha, \lambda)$ is
\begin{equation}\label{rAlpha}
r_\alpha(\lambda) = \frac{\eta + \sqrt{\eta^2 - 4 \lambda}}{2\sqrt{\eta^2 - 4 \lambda}} e^{- \frac{\eta - \sqrt{\eta^2 - 4 \lambda}}{2} \frac{1}{\alpha}} - \frac{\eta - \sqrt{\eta^2 - 4 \lambda}}{2\sqrt{\eta^2 - 4 \lambda}} e^{-\frac{\eta + \sqrt{\eta^2 - 4 \lambda}}{2} \frac{1}{\alpha}}.
\end{equation}

\begin{theorem}\label{ThmReguCase1}
The functions $\{g_\alpha(\lambda), \phi_\alpha(\lambda)\}_\alpha$  in (\ref{gAlphaFun}) based on (\ref{SecondFlow}) satisfy the conditions $(i) - (iii)$ of Proposition~\ref{gAlpha}, which means that we consequently have a regularization method with the procedure (\ref{regularization}) for the linear inverse problem (\ref{OperatorEq}).
\end{theorem}

\begin{proof}
We check all of the three requirements in Proposition \ref{gAlpha}. The first condition obviously holds for $\phi_\alpha(\lambda)$ and $r_\alpha(\lambda)$, defined in (\ref{gAlphaFun}) and (\ref{rAlpha}) respectively.

The second condition can be obtained by using
\begin{eqnarray}\label{ReguThmIneq1}
\left\{\begin{array}{ll}
|r_{\alpha}(\lambda)|\leq \frac{\eta + \sqrt{\eta^2 - 4 \lambda}}{2\sqrt{\eta^2 - 4 \lambda}} = \frac{\eta}{2\sqrt{\eta^2 - 4 \lambda}} + \frac{1}{2} \leq \gamma_1:= \frac{\eta}{2\sqrt{\eta^2 - 4 \|A\|^2}} + \frac{1}{2}, \\
|\phi_{\alpha}(\lambda)|\leq \frac{1}{2\sqrt{\eta^2 - 4 \lambda}} e^{- \frac{\eta - \sqrt{\eta^2 - 4 \lambda}}{2} t} \leq \gamma_2:= \frac{\eta}{2\sqrt{\eta^2 - 4 \|A\|^2}} .
\end{array}\right.
\end{eqnarray}
It remains to bound $\gamma_*$ in Proposition \ref{gAlpha}. By the inequality $1-e^{-at}\leq \sqrt{at}$ for $a> 0$, we obtain
\begin{eqnarray*}
1-  e^{- \frac{\eta - \sqrt{\eta^2 - 4 \lambda}}{2} \frac{1}{\alpha}} \leq \sqrt{\frac{\eta - \sqrt{\eta^2 - 4 \lambda}}{2}} \frac{1}{\sqrt{\alpha}},
\end{eqnarray*}
which implies that
\begin{eqnarray*}
\begin{array}{ll}
\sqrt{\lambda} |g_{\alpha}(\lambda)| = \frac{1}{\sqrt{\lambda}} \left( 1-  \frac{\eta + \sqrt{\eta^2 - 4 \lambda}}{2\sqrt{\eta^2 - 4 \lambda}} e^{- \frac{\eta - \sqrt{\eta^2 - 4 \lambda}}{2} \frac{1}{\alpha}} + \frac{\eta - \sqrt{\eta^2 - 4 \lambda}}{2\sqrt{\eta^2 - 4 \lambda}} e^{-\frac{\eta + \sqrt{\eta^2 - 4 \lambda}}{2} \frac{1}{\alpha}} \right) \\
= \frac{1}{\sqrt{\lambda}} \frac{\eta + \sqrt{\eta^2 - 4 \lambda}}{2\sqrt{\eta^2 - 4 \lambda}} \left( 1-  e^{- \frac{\eta - \sqrt{\eta^2 - 4 \lambda}}{2} \frac{1}{\alpha}} \right) - \frac{1}{\sqrt{\lambda}} \frac{\eta - \sqrt{\eta^2 - 4 \lambda}}{2\sqrt{\eta^2 - 4 \lambda}} \left( 1 -  e^{-\frac{\eta + \sqrt{\eta^2 - 4 \lambda}}{2} \frac{1}{\alpha}} \right) \\
\leq \frac{1}{\sqrt{\lambda}} \frac{\eta + \sqrt{\eta^2 - 4 \lambda}}{2\sqrt{\eta^2 - 4 \lambda}} \left( \sqrt{\frac{\eta - \sqrt{\eta^2 - 4 \lambda}}{2}} \frac{1}{\sqrt{\alpha}} \right) = \frac{1}{\sqrt{\eta^2 - 4 \lambda}}  \sqrt{\frac{\eta + \sqrt{\eta^2 - 4 \lambda}}{2}} \frac{1}{\sqrt{\alpha}} \\
\leq \sqrt{\frac{\eta}{\eta^2 - 4 \|A\|^2}}  \frac{1}{\sqrt{\alpha}}
\end{array}
\end{eqnarray*}
Therefore, the third requirement in Proposition \ref{gAlpha} holds for $g_\alpha(\lambda)$ with
\begin{eqnarray}\label{gammaStara}
\gamma_* = \sqrt{\eta/(\eta^2 - 4 \|A\|^2)}.
\end{eqnarray}
Finally, by the proof above, we see that the upper bound $\bar{\alpha}$ for the affine regularization method with $\{g_\alpha(\lambda), \phi_\alpha(\lambda)\}_\alpha$ can be selected arbitrarily.
\end{proof}

\begin{proposition}\label{ThmQualificationCase1}
For all exponents $p>0$ the monomials $\psi(\lambda) = \lambda^{p}$ are qualifications with the constants
\begin{eqnarray}\label{gamma_a}
\gamma = \left( \frac{p \eta}{e} \right)^p \left( \frac{\eta}{2\sqrt{\eta^2 - 4 \|A\|^2}} + \frac{1}{2} \right)
\end{eqnarray}
for the second order asymptotical regularization method in the overdamped case.
\end{proposition}

\begin{proof}
Set $\xi= (\eta - \sqrt{\eta^2 - 4 \lambda})/2$ in (\ref{Pflimiting1}) and use the following inequality
\begin{eqnarray*}
\frac{\eta - \sqrt{\eta^2 - 4 \lambda}}{2} = \frac{4 \lambda}{2(\eta + \sqrt{\eta^2 - 4 \lambda})} \geq \frac{\lambda}{\eta}
\end{eqnarray*}
and the inequality (\ref{ReguThmIneq1}), and we can derive that
\begin{eqnarray*}
\begin{array}{ll}
\sup\limits_{\lambda\in(0,\|A\|^2]} |r_{\alpha}(\lambda)| \psi(\lambda) \leq \sup\limits_{\lambda\in(0,\|A\|^2]}  \frac{\eta + \sqrt{\eta^2 - 4 \lambda}}{2\sqrt{\eta^2 - 4 \lambda}} e^{- \frac{\eta - \sqrt{\eta^2 - 4 \lambda}}{2} \frac{1}{\alpha}} \lambda^{p} \\ \qquad\qquad \leq
\gamma_1 \sup\limits_{\lambda\in(0,\|A\|^2]} e^{- \frac{\eta - \sqrt{\eta^2 - 4 \lambda}}{2} \frac{1}{\alpha}} \lambda^{p} \\ \qquad\qquad \leq
\gamma_1 \eta^p \sup\limits_{\lambda\in(0,\|A\|^2]} e^{- \frac{\eta - \sqrt{\eta^2 - 4 \lambda}}{2} \frac{1}{\alpha}} \left( \frac{\eta - \sqrt{\eta^2 - 4 \lambda}}{2} \right)^{p} \\ \qquad\qquad \leq
\gamma_1 \eta^p \sup\limits_{\xi\in(0,  (\eta - \sqrt{\eta^2 - 4 \|A\|^2})/2]} e^{- \xi/ \alpha} \left( \xi \right)^{p} \leq \gamma_1 \eta^p \left( \frac{p}{e} \right)^p \alpha^{p} = \gamma \alpha^{p}.
\end{array}
\end{eqnarray*}
Similarly, we have
\begin{eqnarray*}
\begin{array}{ll}
\sup\limits_{\lambda\in(0,\|A\|^2]} |\phi_\alpha(\lambda)| \psi(\lambda) \leq \sup\limits_{\lambda\in(0,\|A\|^2]}  \frac{1}{2\sqrt{\eta^2 - 4 \lambda}} e^{- \frac{\eta - \sqrt{\eta^2 - 4 \lambda}}{2} \frac{1}{\alpha}} \lambda^{p} \\ \qquad \leq
\gamma_2 \sup\limits_{\lambda\in(0,\|A\|^2]} e^{- \frac{\eta - \sqrt{\eta^2 - 4 \lambda}}{2} \frac{1}{\alpha}} \lambda^{p} \leq
\gamma_2 \eta^p \left( \frac{p}{e} \right)^p \alpha^{p} \leq \gamma \alpha^{p},
\end{array}
\end{eqnarray*}
which completes the proof.
\end{proof}

The assertion of Theorem~\ref{ThmReguCase1} and analogues to Proposition~\ref{ThmQualificationCase1} can  be found in the Appendix B for the other values of the constant $\eta>0$ occurring as a parameter in the second order differential equation of problem (\ref{SecondFlow}). In particular, this means the underdamped case (b), as well as the critical damping case (c).


\section{Convergence rates results}
\label{ConvergenceRate}

Under the general assumptions of the previous sections, the rate of convergence of $x(T)\to x^\dagger$ as $T\to\infty$ in the case of precise data, and of $x(T^*(\delta))\to x^\dagger$ as $\delta\to0$ in the case of noisy data, can be arbitrarily slow (cf.~\cite{Schock}) for solutions $x^\dagger$ which are not smooth enough. In order to prove convergence rates, some kind of smoothness assumptions imposed on the exact solution must be employed. Such smoothness assumptions can be expressed by range-type source conditions (cf., e.g.,~\cite{engl1996regularization}), approximate source conditions (cf.~\cite{Hofmann2009}), and variational source conditions occurring in form of
variational inequalities (cf.~\cite{Hofmann2007}). Now, range-type source conditions have the advantage that, in many cases, interpretations in the form of differentiability of the exact solution, boundary conditions, or similar properties are accessible. Hence, we focus in the following on the traditional range-type source conditions only. More precisely, we assume that an element $v_0\in \mathcal{X}$ and numbers $p>0$ and $\rho\geq0$ exist such that
\begin{equation}\label{SourceCondition}
x_0 - x^\dagger = \left( A^* A \right)^{p} v_0 \quad \textrm{~with~} \quad \|v_0\|\leq \rho.
\end{equation}
Moreover, the initial data $\dot{x}_0$ is  supposed to satisfy such source conditions as well, i.e.
\begin{equation}\label{SourceConditionVelocity}
\dot{x}_0 = \left( A^* A \right)^{p} v_1 \quad \textrm{~with~} \quad \|v_1\|\leq \rho.
\end{equation}
For the choice $\dot{x}_0=0$, the condition (\ref{SourceConditionVelocity}) is trivially satisfied. However, following the discussions in Sections 2 and 6, the regularized solutions essentially depend on the value of $\dot{x}_0$. A good choice of $\dot{x}_0$ provides an acceleration of the regularization algorithm. In practice, one can choose a relatively small value of $\dot{x}_0$ to balance the source condition and the acceleration effect.

\begin{proposition}\label{profileFunSecondOrder}
Under the source conditions (\ref{SourceCondition}) and (\ref{SourceConditionVelocity}),
$f(\alpha) = 2 \gamma \rho\, \alpha^{p}$ is a profile function for the second order asymptotical regularization, where the constant $\gamma$ is defined in (\ref{gamma_a}).
\end{proposition}

\begin{proof}
Combining the formulas (\ref{errorBounds2}), (\ref{SourceCondition}) and (\ref{SourceConditionVelocity}) yields
\begin{eqnarray*}
\begin{array}{ll}
\|x(1/\alpha) - x^\dagger\| \leq \|r_\alpha(A^* A) (x_0 - x^\dagger)\| + \|\phi_\alpha(A^* A)\dot{x}_0\| \\ \qquad\qquad
\leq \|r_\alpha(A^*A) \left( A^* A \right)^{p} v_0 \| + \|\phi_\alpha(A^*A) \left( A^* A \right)^{p} v_1 \| \\ \qquad\qquad
\leq \rho \sup \limits _{0\leq \lambda \leq \|A\|^2} |r_\alpha(\lambda)|\, \lambda^{p} + \rho \sup \limits_{0\leq \lambda \leq \|A\|^2} |\phi_\alpha(\lambda)|\, \lambda^{p} \leq 2 \gamma \rho \,\alpha^{p}.
\end{array}
\end{eqnarray*}
This proves the proposition.
\end{proof}

\begin{theorem}\label{ThmPriori}
(A priori choice of the regularization parameter) If the terminating time $T^*$ of the second order flow (\ref{SecondFlow}) is selected by the a priori parameter choice
\begin{equation}\label{Tpriori}
T^* (\delta) =c_0 \rho^{2/(2p+1)} \,\delta^{-2/(2p+1)}
\end{equation}
with the constant $c_0=(2\gamma)^{2/(2p+1)}$, then we have the error estimate for $\delta\in(0,\delta_0]$
\begin{equation}\label{ErrorEstimatePriori}
\| x(T^*) - x^\dagger \| \leq c \rho^{1/(2p+1)}\, \delta^{2p/(2p+1)},
\end{equation}
where the constant $c=(1 + \gamma_*) (2\gamma)^{1/(2p+1)}$ and $\delta_0 = 2\gamma\rho \eta^{2p+1}$.
\end{theorem}

\begin{proof}
By the discussion in Section 2, we choose the value of $\alpha_*$ such that \linebreak $\delta = \Theta(\alpha_*) = \sqrt{\alpha_*} f(\alpha_*) = 2 \gamma\rho \alpha^{p+1/2}_*$. By solving this equation we directly obtain $\alpha_*=(2\gamma)^{-2/(2p+1)}$ $\rho^{-2/(2p+1)} \delta^{2/(2p+1)}$. Setting $T^*=1/\alpha_*$ and using the estimate (\ref{Estimator3}), this gives the relations (\ref{Tpriori}) and
\begin{eqnarray*}
\begin{array}{ll}
\| x(T^*) - x^\dagger \| = e(x^\dagger, \alpha, \delta) \leq (1 + \gamma_*) f(\alpha_*) = (1 + \gamma_*) 2 \gamma \rho (T^*)^{-p} \\ \qquad\qquad\qquad
= \left\{ (1 + \gamma_*) (2\gamma)^{1/(2p+1)} \right\} \rho^{1/(2p+1)} \delta^{2p/(2p+1)}.
\end{array}
\end{eqnarray*}
Finally, we use the inequality $\alpha_*\leq \bar{\alpha}=\eta^2$ to get the bound $\delta_0$ (the upper bound $\bar{\alpha}=\eta^2$ is required for the affine regularization (\ref{regularization}) in both the underdamped and critical cases; see the appendix for details).
\end{proof}

In practice, the stopping rule in (\ref{Tpriori}) is not realistic, since a good terminating time point $T^*$ requires knowledge of $\rho$ (a characteristic of unknown exact solution). Such knowledge, however, is not necessary in the case of a posteriori parameter choices. In the following two subsections, we consider two types of discrepancy principles for choosing the terminating time point a posteriori.

\subsection{Morozov's conventional discrepancy principle}

In our setting, Morozov's conventional discrepancy principle means searching for values $T>0$ satisfying the equation
\begin{eqnarray}\label{discrepancy1}
\chi(T):= \|A x(T)-y^\delta\| - \tau \delta=0,
\end{eqnarray}
where $\tau> \gamma_1 \geq1$ is a constant, and the number $\gamma_1$ is defined in Proposition \ref{gAlpha}.

\begin{lemma}\label{Rootdiscrepancy}
If $\|Ax_0-y^\delta\|> \tau \delta$, then the function $\chi(T)$ has at least one root.
\end{lemma}

\begin{proof}
The continuity of $\chi(T)$ is obvious according to Theorem \ref{ExsisteceUnique}. On the other hand, from (\ref{ResidualErrorLimit}) and the assumption of the lemma, we conclude that
\begin{eqnarray*}\label{TwoLimits}
\lim_{T\to \infty} \chi(T) \leq (1-\tau)\delta  <0 \quad \textrm{~and~} \quad \chi(0) = \|Ax_0-y^\delta\| - \tau \delta >0,
\end{eqnarray*}
which implies the existence of the root of $\chi(T)$.
\end{proof}

\begin{theorem}\label{ThmPosteriori}
(A posteriori choice I of the regularization parameter) Suppose that $\|Ax_0-y^\delta\|> \tau \delta$ and the source conditions (\ref{SourceCondition}) and (\ref{SourceConditionVelocity}) hold. If the terminating time $T^*$ of the second order flow (\ref{SecondFlow}) is chosen according to the discrepancy principle (\ref{discrepancy1}), we have for any $\delta\in(0, \delta_0]$ and $p> 0$ the error estimates
\begin{equation}\label{ErrorEstimatePrioriT}
T^* \leq C_0 \rho^{2/(2p+1)} \delta^{-2/(2p+1)}
\end{equation}
and
\begin{equation}\label{ErrorEstimatePriori}
\| x(T^*) - x^\dagger \| \leq C_1 \delta^{2p/(2p+1)},
\end{equation}
where $\delta_0$ is defined in the Theorem \ref{ThmPriori}, $C_0:= (\tau-\gamma_1)^{-2/(2p+1)} (2\gamma)^{2/(2p+1)}$, and $C_1:= \left( \tau + \gamma_1 \right)^{2p/(2p+1)} (\gamma_1+\gamma_2)^{1/(2p+1)} + \gamma_* (\tau-\gamma_1)^{-1/(2p+1)} (2\gamma)^{1/(2p+1)}$.
\end{theorem}

\begin{proof}
Using the moment inequality $\|B^p u\| \leq \|B^q u\|^{p/q} \|u\|^{1-p/q}$ and the source conditions (\ref{SourceCondition})-(\ref{SourceConditionVelocity}), we deduce that
\begin{eqnarray}\label{PosterioriProofIneq1}
\begin{array}{ll}
\| r_{\alpha}(A^*A^*) (x_0 - x^\dagger) + \phi_{\alpha}(A^*A^*) \dot{x}_0 \|  \\ \quad
= \| (A^*A)^{(p+1/2)} \left( r_{\alpha}(A^*A) v_0 + \phi_{\alpha}(A^*A) v_1 \right) \| \\ \quad
\leq \| (A^*A)^{(p+1/2)} \left( r_{\alpha}(A^*A) v_0 + \phi_{\alpha}(A^*A) v_1 \right) \|^{2p/(2p+1)} \\ \qquad\qquad
\cdot \| r_{\alpha}(A^*A) v_0 + \phi_{\alpha}(A^*A) v_1 \|^{1/(2p+1)} \\ \quad
\leq \| A r_{\alpha}(A^*A) (x_0 - x^\dagger) + A \phi_{\alpha}(A^*A^*) \dot{x}_0 \|^{2p/(2p+1)} \\ \qquad\qquad
\cdot \left( \| r_{\alpha}(A^*A) v_0 \| + \| \phi_{\alpha}(A^*A) v_1 \| \right)^{1/(2p+1)}.
\end{array}
\end{eqnarray}
Since the terminating time $T^*$ is chosen according to the discrepancy principle (\ref{discrepancy1}), we derive that
\begin{eqnarray}\label{PosterioriProofIneq2}
\begin{array}{ll}
\tau \delta = \|A x(T^*)-y^\delta\| \\ \qquad
= \left\| A r_{1/T^*}(A^*A) (x_0- x^\dagger) + A \phi_{1/T^*}(A^*A) \dot{x}_0 - r_{1/T^*}(A^*A) (y^\delta-y) \right\| \\ \qquad
\geq \| A r_{1/T^*}(A^*A) (x_0- x^\dagger) + A \phi_{1/T^*}(A^*A) \dot{x}_0\| - \|r_{1/T^*}(A^*A) (y^\delta-y) \|
\end{array}
\end{eqnarray}
Now we combine the estimates (\ref{PosterioriProofIneq1}) and (\ref{PosterioriProofIneq2}) to obtain, with the source conditions, that
\begin{eqnarray}\label{PosterioriProofIneq3}
\begin{array}{ll}
\| r_{\alpha}(A^*A^*) (x_0 - x^\dagger) + \phi_{\alpha}(A^*A^*) \dot{x}_0 \|  \\ \quad
\leq \| A r_{\alpha}(A^*A) (x_0 - x^\dagger) + A \phi_{\alpha}(A^*A^*) \dot{x}_0 \|^{2p/(2p+1)} \\ \qquad\qquad
\cdot \left( \| r_{\alpha}(A^*A) v_0 \| + \| \phi_{\alpha}(A^*A) v_1 \| \right)^{1/(2p+1)} \\ \quad
\leq \left( \tau \delta + \|r_{1/T^*}(A^*A) (y^\delta-y) \| \right)^{2p/(2p+1)} \left( (\gamma_1+\gamma_2)\rho \right)^{1/(2p+1)} \\ \quad
\leq c_1 \rho^{1/(2p+1)} \delta^{2p/(2p+1)}
\end{array}
\end{eqnarray}
where $c_1:= \left( \tau + \gamma_1 \right)^{2p/(2p+1)} (\gamma_1+\gamma_2)^{1/(2p+1)}$.

On the other hand, in a similar fashion to (\ref{PosterioriProofIneq2}), it is easy to show that
\begin{eqnarray*}
\begin{array}{ll}
\tau \delta \leq \| A r_{1/T^*}(A^*A) (x_0- x^\dagger) + A \phi_{1/T^*}(A^*A) \dot{x}_0\| + \|r_{1/T^*}(A^*A) (y^\delta-y) \| \\ \quad
\leq \| A r_{1/T^*}(A^*A) (x_0- x^\dagger) + A \phi_{1/T^*}(A^*A) \dot{x}_0\| + \gamma_1 \delta.
\end{array}
\end{eqnarray*}
If we combine the above inequality with the source conditions (\ref{SourceCondition})-(\ref{SourceConditionVelocity}) and the qualification inequality (\ref{QualificationDef}), we obtain
\begin{eqnarray*}
\begin{array}{ll}
(\tau-\gamma_1)\delta \leq \| A r_{1/T^*}(A^*A) (x_0- x^\dagger) + A \phi_{1/T^*}(A^*A) \dot{x}_0\| \\ \quad
\leq  \| (A^*A)^{p+1/2} r_{1/T^*}(A^*A) v_0 \| + \|(A^*A)^{p+1/2} \phi_{1/T^*}(A^*A) v_1\| \leq 2\rho \gamma (T^*)^{-(p+1/2)},
\end{array}
\end{eqnarray*}
which yields the estimate (\ref{ErrorEstimatePrioriT}).
Finally, using (\ref{ErrorEstimatePrioriT}) and (\ref{PosterioriProofIneq3}), we conclude that
\begin{eqnarray*}
\begin{array}{ll}
\| x(T^*) - x^\dagger \| \leq \| r_{\alpha}(A^*A^*) (x_0 - x^\dagger) + \phi_{\alpha}(A^*A^*) \dot{x}_0 \| + \gamma_* \sqrt{T^*} \delta \\ \quad
\leq c_1 \rho^{1/(2p+1)} \delta^{2p/(2p+1)} + \gamma_* \sqrt{C_0} \rho^{1/(2p+1)} \delta^{2p/(2p+1)} = C_1 \rho^{1/(2p+1)} \delta^{2p/(2p+1)}.
\end{array}
\end{eqnarray*}
This completes the proof.
\end{proof}

\begin{remark}
If the function $\chi(T)$ has more than one root, we recommend selecting $T^*$ from the rule
\begin{eqnarray*}
\chi(T^*) = 0< \chi(T), \quad \forall T<T^*.
\end{eqnarray*}
In other words, $T^*$ is the first time point for which the size of the residual $\|A x(T)-y^\delta\|$ has about the order of the data error. By Lemma~\ref{Rootdiscrepancy} such $T^*$ always exists.

It is easy to show that $\chi(T)$ is bounded by a decreasing function as the proof of Proposition \ref{PropositionDiscrepancyFun} below will show. Roughly speaking,  the trend of $\chi(T)$ is to be a decreasing function, where oscillations may occur, and we refer to Figure~\ref{Fig:damping}. On the other hand, one can anticipate that the more oscillations of the discrepancy function $\chi(T)$ occur, the smaller the damping parameter $\eta$ is. This is an expected result due to the behaviour of damped Hamiltonian systems.
\end{remark}

\begin{figure}[!htb]
\centering
\includegraphics[width=1.\textwidth]{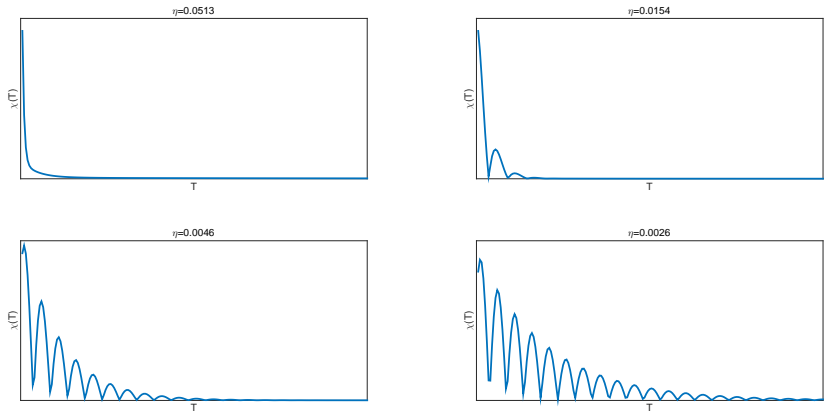}
{\footnotesize
\caption{The behaviour of $\chi(T)$ from (\ref{discrepancy1}) with different damping parameters $\eta$.}}
\label{Fig:damping}
\end{figure}

\subsection{The total energy discrepancy principle}

For presenting a newly developed discrepancy principle, we introduce the total energy discrepancy function as follows:
\begin{eqnarray}\label{discrepancy0}
\chi_{te}(T):= \|A x(T)-y^\delta\|^2+ \|\dot{x}(T)\|^2 - \tau^2 \delta^2,
\end{eqnarray}
where $\tau> \gamma_1$ as before.

\begin{proposition}\label{PropositionDiscrepancyFun}
The function $\chi_{te}(T)$ is continuous and monotonically non-increasing. If $\|A x_0-y^\delta\|^2+ \|\dot{x}_0\|^2> \tau^2 \delta^2$, then  $\chi_{te}(T)=0$  has a unique solution.
\end{proposition}

\begin{proof}
The continuity of $\chi_{te}(T)$ is obvious according to Theorem~\ref{ExsisteceUnique}. The non-growth of $\chi_{te}(T)$ is straight-forward according to $\dot{\chi}_{te} = - 2 \eta \|\dot{x}\|^2$. Furthermore, from (\ref{ResidualErrorLimit}), (\ref{RateLimits}) and the assumption of the proposition, we derive that
\begin{eqnarray}\label{Limitchi2}
\lim \limits_{T\to \infty} \chi_{te}(T) \leq \delta^2(1-\tau^2)  <0,
\end{eqnarray}
and that, moreover,  $\chi_{te}(0) = \|A x_0-y^\delta\|^2+ \|\dot{x}_0\|^2 - \tau^2 \delta^2 >0$. This implies the existence of roots for $\chi_{te}(T)$.

Finally, let us show that $\chi_{te}(T)$ has a unique solution. We prove this by contradiction. Since $\chi_{te}(T)$ is a non-increasing function, a number $T_0$ exists so that $\chi_{te}(T)=0$ for $T\in [T_0,T_0+\varepsilon]$ with some positive $\varepsilon>0$. This means that $\dot{\chi}_{te}(T)=- 2\eta \|\dot{x}\|^2\equiv0$ in $(T_0,T_0+\varepsilon)$. Hence, $\ddot{x}\equiv0$ in $(T_0,T_0+\varepsilon)$. Using the equation (\ref{SecondFlow}) we conclude that for all $T>T_0$: $x(T)\equiv x(T_0)$. Since $\chi_{te}(T_0)=0$, we obtain that $\chi_{te}(T)\equiv0$ for $T>T_0$, which implies that $\lim \limits_{T\to \infty} \chi_{te}(T) = 0$. This contradicts the fact in (\ref{Limitchi2}).
\end{proof}

\begin{theorem}\label{ThmPosteriori20}
(A posteriori choice II of the regularization parameter) Assume that $\|A x_0-y^\delta\|^2+ \|\dot{x}_0\|^2> \tau^2 \delta^2$ and a positive number $\delta_1$ exists such that for all $\delta\in(0,\delta_1]$, the unique root $T^*$ of $\chi_{te}(T)$ satisfies the inequality $\|A x(T^*)-y^\delta\|\geq \tau_1 \delta$, where $\tau_1> \gamma_1$ is a constant, independent of $\delta$. Then, under the source conditions (\ref{SourceCondition}) and (\ref{SourceConditionVelocity}), for any $\delta\in(0, \delta_0]$ and $p> 0$ we have the error estimates
\begin{equation}\label{ErrorEstimatePrioriT2}
T^* \leq C_0 \rho^{2/(2p+1)}\, \delta^{-2/(2p+1)}, \quad \| x(T^*) - x^\dagger \| \leq C_1\, \delta^{2p/(2p+1)},
\end{equation}
where $\delta_0$ is defined in the Theorem \ref{ThmPriori}, and constants $C_0$ and $C_1$ are the same as in Theorem \ref{ThmPosteriori}.
\end{theorem}

\begin{proof}
The proof can be done along the lines and using the tools of the proof of Theorem~\ref{ThmPosteriori}.
\end{proof}

In the simulation Section 7.1, we will computationally show that the assumptions occurring in the above theorem can happen in practice. Empirically, when the value of the initial velocity is not too small ($\|\dot{x}_0\|>0$) or the noise is small enough ($\delta\ll 1$), the additional assumption $\|A x(T^*)-y^\delta\|\geq \tau_1 \delta$ in Theorem~\ref{ThmPosteriori20} always holds.


\section{A novel iterative regularization method}

Roughly speaking, the second order evolution equation (\ref{SecondFlow}) with an appropriate numerical discretization scheme for the artificial time variable yields a concrete second order iterative method. Just as with the Runge-Kutta integrators~\cite{Rieder-2005} or the exponential integrators~\cite{Hochbruck-1998} for solving first order equations, the damped symplectic integrators are extremely attractive for solving second order equations, since the schemes are closely related to the canonical transformations~\cite{Hairer-2006}, and the trajectories of the discretized second flow are usually more stable.

The simplest discretization scheme should be the Euler method. Denote by $v=\dot{x}$, and consider the following Euler iteration
\begin{equation}\label{symplecticEuler}
\left\{\begin{array}{l}
x^{k+1} = x^{k} + \Delta t_k v^{k} , \\
v^{k+1} = v^{k} + \Delta t_k \left( A^*(y^\delta-Ax^{k+1}) - \eta v^{k} \right) , \\
x^{0}=x_0, v^{0}=\dot{x}_0.
\end{array}\right.
\end{equation}
By elementary calculations, scheme (\ref{symplecticEuler}) expresses the form of following three-term semi-iterative method
\begin{equation}\label{semiiterative}
x^{k+1}=x^{k} + \mu_k \left( x^{k}-x^{k-1} \right) + \omega_k A^*(y^\delta-Ax^{k})
\end{equation}
with a specially defined parameters $\omega_k=\Delta t_k$ and $\mu_k=1-\Delta t_k \eta$. It is well known that the semi-iterative method (\ref{semiiterative}), equipped with an appropriate stopping rule, yields order optimal regularization method with (asymptotically) much fewer iterations than the calssical Landweber iteration~\cite[\S~6.2]{engl1996regularization}.

In this paper, we develop a new iterative regularization method based on the St\"{o}rmer-Verlet method, which also belongs to the family of symplectic integrators and takes the form
\begin{equation}\label{symplectic}
\left\{\begin{array}{l}
v^{k+\frac{1}{2}} = v^{k} + \frac{\Delta t}{2} \left( A^*(y^\delta-Ax^{k}) - \eta v^{k+\frac{1}{2}} \right), \\
x^{k+1} = x^{k} + \Delta t v^{k+\frac{1}{2}} , \\
v^{k+1} = v^{k+\frac{1}{2}} + \frac{\Delta t}{2} \left( A^*(y^\delta-Ax^{k+1}) - \eta v^{k+\frac{1}{2}} \right) , \\
x^{0}=x_0, v^{0}=\dot{x}_0.
\end{array}\right.
\end{equation}

\begin{proposition}\label{ConvergenceAlgorithm}
For any fixed damping parameter $\eta$, if the step size is chosen by
\begin{equation}\label{StepSize}
\Delta t \leq \min\left( \sqrt{2}/\|A\|, 2/\eta \right),
\end{equation}
then, the scheme (\ref{symplectic}) is convergent. Consequently, for any fixed $T$, there exists a pair of parameters $(k,\Delta t)$, satisfying $k\Delta t=T$ and the condition (\ref{StepSize}), such that $x^{k}= x(T) + \mathcal{O}(\Delta t^2)$ as $\Delta t\to0$. Here $x^{k}$ and $x(\cdot)$ are solutions to (\ref{symplectic}) and (\ref{SecondFlow}) respectively.
\end{proposition}

\begin{proof}
Denote by $\mathbf{z}=(x,v)^T$, and rewritten (\ref{symplectic}) as
\begin{equation}\label{symplecticMatrix}
\mathbf{z}^{k+1} = B \mathbf{z}^{k} + \frac{\Delta t}{2+ \Delta t \eta} \mathbf{b},
\end{equation}
where $\mathbf{b}=[\Delta t; 2I- \frac{\Delta t^2}{2} A^* A] A^* y_\delta$ and
\begin{eqnarray*}\label{coefficients}
B  = \left[ \begin{array}{cc}
I - \frac{\Delta t^2}{2+ \Delta t \eta} A^* A & \frac{2\Delta t}{2+ \Delta t \eta} I \\
- \frac{\Delta t}{2(2+ \Delta t \eta)} \left( 4I - \Delta t^2 A^* A \right) A^* A \, & \, \frac{2-\Delta t \eta}{2+ \Delta t \eta} I - \frac{\Delta t^2}{2+ \Delta t \eta} A^* A
\end{array}\right].
\end{eqnarray*}

By Taylor's theorem and the finite difference formula, it is not difficult to show the consistency of the scheme (\ref{symplectic}). It is well known that boundedness implies the convergence of consistent schemes for any problem~\cite{Tadmor}, hence, it suffices to show the boundedness of the scheme (\ref{symplectic}). The asymptotical behaviour $x^{k}= x(T) + \mathcal{O}(\Delta t^2)$ follows from the convergence result and the second order of the St\"{o}rmer-Verlet method. Furthermore, a sufficient condition for the boundedness of the iterative algorithm (\ref{symplectic}) is that the operator $B$ is non-expansive. Hence, it is necessary to prove that $\|B\|_2\leq1$.

Using the singular value decomposition, we have $A^* A= \Phi \Lambda \Phi^T$, where $\Phi$ is a unitary matrix and $\Lambda= \textrm{diag} (\lambda_i)$, where $\lambda_i\geq0, i=1, ...,n$.

By the elementary calculations, we obtain that the eigenvalues of $B$ equal
\begin{equation}\label{symplecticMatrix}
\quad \mu^{\pm}_i = \frac{1}{2+ \Delta t \eta} \left\{ \left( 2 - \Delta t^2 \lambda_i \right) \pm \sqrt{\left( 2 - \Delta t^2 \lambda_i \right)^2 - \left( 4 - \Delta t^2 \eta^2 \right)} \right\}, i=1, ...,n.
\end{equation}

Denote by $i_*$ the index of $\lambda_{i_*}$, corresponding to the maximal absolute value of $\mu^{\pm}_i$, i.e.
\begin{eqnarray*}
|\mu_{max}| = \frac{1}{2+ \Delta t \eta} \max_{\pm} \left| \left( 2 - \Delta t^2 \lambda_{i_*} \right) \pm \sqrt{\left( 2 - \Delta t^2 \lambda_{i_*} \right)^2 - \left( 4 - \Delta t^2 \eta^2 \right)} \right| .
\end{eqnarray*}

There are three possible cases here: the overdamped case ($ \left( 2 - \Delta t^2 \lambda_{i_*} \right)^2 > 4 - \Delta t^2 \eta^2$), the underdamped case ($\left( 2 - \Delta t^2 \lambda_{i_*} \right)^2 < 4 - \Delta t^2 \eta^2$), and the critical damping case ($\left( 2 - \Delta t^2 \lambda_{i_*} \right)^2 = 4 - \Delta t^2 \eta^2$).

Let us consider these cases respectively. For the chosen time step size $\Delta t$ in (\ref{StepSize}), we have $2 - \Delta t^2 \lambda_{i_*}\geq0$. Therefore, for the overdamped case,
\begin{eqnarray*}
|\mu_{max}| = \frac{1}{2+ \Delta t \eta}  \left\{ \left( 2 - \Delta t^2 \lambda_{i_*} \right) + \sqrt{\left( 2 - \Delta t^2 \lambda_{i_*} \right)^2 - \left( 4 - \Delta t^2 \eta^2 \right)} \right\}.
\end{eqnarray*}

Define $2 - \Delta t^2 \lambda_{i_*} =a \sqrt{4 - \Delta t^2 \eta^2}$ with $a>1$ (by the condition (\ref{StepSize}). It holds that $4 - \Delta t^2 \eta^2\geq0$), and we have
\begin{eqnarray}\label{algorithmProof1}
|\mu_{max}| = \frac{(a+\sqrt{a^2-1})\sqrt{4 - \Delta t^2 \eta^2}}{2+ \Delta t \eta} = \frac{(a+\sqrt{a^2-1})(2-\Delta t^2 \lambda_{i_*})}{a(2+ \Delta t \eta)}.
\end{eqnarray}
Note that
\begin{equation}\label{algorithmProof2}
\Delta t \eta= \sqrt{4 - \left( \frac{2-\Delta t^2 \lambda_{i_*}}{a} \right)^2}.
\end{equation}
Combine (\ref{algorithmProof1}) and (\ref{algorithmProof2}) to obtain
\begin{eqnarray*}
|\mu_{max}| = \frac{(a+\sqrt{a^2-1})(2-\Delta t^2 \lambda_{i_*})}{2a+ \sqrt{4a^2 - \left( 2-\Delta t^2 \lambda_{i_*} \right)^2}} =  \frac{(a+\sqrt{a^2-1})(1-\frac{\Delta t^2 \lambda_{i_*}}{2})}{a+ \sqrt{a^2 - \left( 1-\frac{\Delta t^2 \lambda_{i_*}}{2} \right)^2}}
\leq 1-\frac{\Delta t^2 \lambda_{i_*}}{2} \leq 1.
\end{eqnarray*}

Now, consider the underdamped case. In this case, the complex eigenvalue $\mu_{max}$ satisfies
\begin{eqnarray*}
|\mu_{max}|^2 = \frac{1}{(2+ \Delta t \eta)^2}  \left\{ \left( 2 - \Delta t^2 \lambda_{i_*} \right)^2 + \left[ \left( 4 - \Delta t^2 \eta^2 \right)- \left( 2 - \Delta t^2 \lambda_{i_*} \right)^2 \right]  \right\},
\end{eqnarray*}
which implies that
\begin{eqnarray*}
|\mu_{max}|^2 = \frac{4 - \Delta t^2 \eta^2}{(2+ \Delta t \eta)^2} = \frac{2 - \Delta t \eta}{2+ \Delta t \eta} < 1.
\end{eqnarray*}

Finally, consider the critical damping case. Similarly, we have $|\mu_{max}|= \sqrt{\frac{2 - \Delta t \eta}{2+ \Delta t \eta}}<1$, which yields the desired result.

\end{proof}

At the end of this second, we show the convergence rate of the scheme (\ref{symplectic}).

\begin{theorem}\label{ConvergenceRateDiscrete}
Under the assumptions of Theorem \ref{ThmPosteriori} or \ref{ThmPosteriori20}, if the time step size is chosen by $\Delta t=C_t \delta^{p/(2p+1)}$, then for any $\delta\in(0, \delta_s]$ and $p> 0$ we have the convergence rate
\begin{equation}\label{ConvergenceRateDiscreteEstimate}
\|x^{k^*} - x^\dagger\| = \mathcal{O}(\delta^{2p/(2p+1)}),
\end{equation}
where $k^*=T^*/\Delta t$ , $\delta_s= \min\left\{ (\sqrt{2}/C_t)^{2+1/p} \|A\|^{-2-1/p}, (\sqrt{2}/C_t)^{2+1/p} \eta^{-2-1/p}, \delta_0\right\}$, $C_t=(T^*/\delta^{p/(2p+1)})/\lfloor T^*/\delta^{p/(2p+1)} \rfloor$, and $T^*$ is the root of (\ref{discrepancy1}) or (\ref{discrepancy0}).  Here, $\lfloor \cdot \rfloor$ denotes the standard floor function and $\delta_0$ is defined in the Theorem \ref{ThmPriori}.
\end{theorem}

\begin{proof}
It follows from Proposition \ref{ConvergenceAlgorithm} that the choice $\Delta t=C_t \delta^{p/(2p+1)}$ yields
\begin{equation*}\label{ProofEnd1}
\|x^{k^*} - x(T^*)\| \leq C \delta^{2p/(2p+1)}
\end{equation*}
with some constant $C$ for all $\delta\in(0, \delta_s]$.

Combine the above inequality and the results in theorems \ref{ThmPosteriori} and \ref{ThmPosteriori20} to obtain
\begin{eqnarray*}
\|x^{k^*} - x^\dagger\| \leq \|x^{k} - x(T^*)\| + \| x(T^*) - x^\dagger \| \leq C_r\, \delta^{2p/(2p+1)}
\end{eqnarray*}
for some constant $C_r$. This gives the estimate (\ref{ConvergenceRateDiscreteEstimate}).
\end{proof}


\section{Numerical simulations}
In this section, we present some numerical results for the following integral equation
\begin{equation}\label{IntegralEq}
Ax(s):= \int^1_0 K(s,t) x(t) dt = y(s), \quad K(s,t)=s(1-t)\chi_{s\leq t} + t(1-s)\chi_{s> t}.
\end{equation}
If we choose $\mathcal{X}=\mathcal{Y}=L^2[0,1]$, the operator $A$ is compact, selfadjoint and injective. It is well known that the integral equation (\ref{IntegralEq}) has a solution $x=-y''$ if $y\in H^2[0,1]\cap H^1_0[0,1]$. Moreover, the operator $A$ has the eigensystem $A u_j = \sigma_j u_j$, where $\sigma_j=(j\pi)^{-2}$ and $u_j(t)=\sqrt{2}\sin(j\pi t)$. Furthermore, using the interpolation theory (see e.g.~\cite{Lions-1972}) it is not difficult to show that for $4p-1/2 \not\in \mathbf{N}$
\begin{eqnarray*}
R ((A^*A)^{p}) = \left\{ x\in H^{4p}[0,1]:~x^{2l}(0)=x^{2l}(1)=0,~l=0,1,...,\lfloor 2p-1/4 \rfloor \right\}.
\end{eqnarray*}

In general, a regularization procedure becomes numerically feasible only after an appropriate discretization. Here, we apply the linear finite elements to solve (\ref{IntegralEq}). Let $\mathcal{Y}_n$ be the finite element space of piecewise linear functions on a uniform grid with step size $1/(n-1)$. Denote by $P_n$ the orthogonal projection operator acting from $\mathcal{Y}$ into $\mathcal{Y}_n$. Define $A_n:= P_n A$ and $\mathcal{X}_n:= A^*_n \mathcal{Y}_n$. Let $\{\phi_j\}^n_{j=1}$ be a basis of the finite element space $\mathcal{Y}_n$, then, instead of the original problem (\ref{IntegralEq}), we solve the following system of linear equations
\begin{equation}\label{FiniteModel}
A_n x_n = y_n,
\end{equation}
where $[A_n]_{ij}= \int^1_0 \left(  \int^1_0 k(s,t) \phi_i(s) ds \right) \phi_j(t) dt$ and $[y_n]_{j} = \int^1_0 y(t) \phi_j(t) dt$.

As shown in~\cite{engl1996regularization}, the finite dimensional projection error $\epsilon_n:=\|(I-P_n)A\|$ plays an important role in the convergence rates analysis. For the compact operator $A$, $\epsilon_n\to0$ as $n\to\infty$. Moreover, if the noise level $\delta\to0$ slowly enough as $n\to\infty$, the quality $\epsilon_n$ has no influence and we obtain the same convergence rates as in theorems \ref{ThmPosteriori} and \ref{ThmPosteriori20}.

Uniformly distributed noises with the magnitude $\delta'$ are added to the discretized exact right-hand side:
\begin{equation}\label{Data}
[y^\delta_n]_j := \left[ 1 + \delta' \cdot(2 \textrm{Rand}(x) -1) \right] \cdot [y_n]_j, \quad j=1, ..., n,
\end{equation}
where $\textrm{Rand}(x)$ returns a pseudo-random value drawn from a uniform distribution on [0,1]. The noise level of measurement data is calculated by $\delta=\|y^\delta_n - y_n\|_2$, where $\|\cdot\|_2$ denotes the standard vector norm in $\mathbf{R}^n$.

To assess the accuracy of the approximate solutions, we define the $L^2$-norm relative error for an approximate solution $x^{k^*}_n$ ($k^*=\lfloor T^*/\Delta t\rfloor$): $$\textrm{L2Err}:= \|x^{k^*}_n - x^\dagger\|_{L^2[0,1]}/\|x^\dagger\|_{L^2[0,1]},$$ where $x^\dagger$ is the exact solution to the corresponding model problem.

\subsection{Influence of parameters}
The purpose of this subsection is to explore the dependence of the solution accuracy and the convergence rate on the initial data $(x_0,\dot{x}_0)$, damping parameter $\eta$ and the discrepancy functions $\chi$ and $\chi_{te}$, and thus to give a guide on the choices of them in practice.

In this subsection, we solve integral equation (\ref{IntegralEq}) with the exact right-hand side $y=s(1-s)$. Then, the exact solution $x^\dagger=2$, and $x^\dagger\in R((A^*A)^{p})$ for all $p<1/8$. Denote by ``DP'' and ``TEDP'' the newly developed iterative scheme (\ref{symplectic}) equipped with the Morozov's conventional discrepancy function $\chi(T)$ and the total energy discrepancy functions $\chi_{te}(T)$ respectively.

The results about the influence of the solution accuracy (L2Err) and the convergence rate (iteration numbers $k^*(\delta)$) on the initial data $(x_0,\dot{x}_0)$ are given in Tables \ref{Tab:x0} and  \ref{Tab:velocity0} respectively. As we can see, both the initial data $x_0$ and $\dot{x}_0$ influence the solution accuracy as well as the convergence rate. Moreover, when the value of the damping parameter is not too small (see Tab. \ref{Tab:velocity0} and Tab. \ref{Tab:Damping}) the results (both solution accuracy and convergence rate) by the methods ``DP'' and ``TEDP'' almost coincide with each other. This result verifies the rationality of the assumption in Theorem \ref{ThmPosteriori20}.

\begin{table}[!tbhp]
{\footnotesize
\caption{The dependence of the solution accuracy and the convergence rate on the initial data $x_0$. $\Delta t=19.4946, \eta=2.5648\times 10^{-4}, \dot{x}_0=0, \tau=2, p=0.1125, \tau_{te}=1.1\times \delta^{4p/(4p+1)}$.}
\begin{center}
\begin{tabular}{|c|c|c|c|c|c|} \hline
\multirow{2}{*}{$\delta$ ($\delta'$)} &
\multirow{2}{*}{$x_0$} &
\multicolumn{2}{c|}{DP} &
\multicolumn{2}{c|}{TEDP} \\
\cline{3-6}
 &  & $k^*(\delta)$ & \textrm{L2Err}& $k^*(\delta)$ & \textrm{L2Err}  \\ \hline
7.1191e-04 (1e-03) & 0 & 1295 & 0.2214 & 2033 & 0.2894  \\
7.1337e-04 (1e-03) & -1 & 1354 & 0.2284 & 2210 & 0.2825  \\
7.0838e-04 (1e-03) & 1 & 1108 & 0.1965 & 1777 & 0.1414  \\
7.1602e-05 (1e-04) & 0 & 2397 & 0.0688 & 2958 & 0.0653  \\
8.0523e-05 (1e-04) & -1 & 2525 & 0.0983 & 3076 & 0.0920  \\
7.8903e-05 (1e-04) & 1 & 1600 & 0.0449 & 2643 & 0.0431 \\ \hline
\end{tabular}\label{Tab:x0}
\end{center}
}
\end{table}

\begin{table}[!tbhp]
{\footnotesize
\caption{The dependence of the solution accuracy and the convergence rate on the initial data $\dot{x}_0$. $\Delta t=19.4946, \eta=0.0154, x_0=1, \tau=2, p=0.1125, \tau_{te}=1.1\times \delta^{4p/(4p+1)}$.}
\begin{center}
\begin{tabular}{|c|c|c|c|c|c|} \hline
\multirow{2}{*}{$\delta$ ($\delta'$)} &
\multirow{2}{*}{$\dot{x}_0$} &
\multicolumn{2}{c|}{DP} &
\multicolumn{2}{c|}{TEDP} \\
\cline{3-6}
 &  & $k^*(\delta)$ & \textrm{L2Err}& $k^*(\delta)$ & \textrm{L2Err}  \\ \hline
6.3541e-04 (1e-03) & 0 & 155 & 0.1062 & 156 & 0.1061  \\
7.3673e-04 (1e-03) & -0.01 & 374 & 0.1576 & 376 & 0.1574 \\
7.1484e-04 (1e-03) & 0.01 & 39 & 0.0443 & 39 & 0.0443 \\
8.0763e-05 (1e-04) & 0 & 4559 & 0.0675 & 4561 & 0.0675 \\
7.3174e-05 (1e-04) & -0.01 & 12466 & 0.0953 & 12467 & 0.0953  \\
7.6200e-05 (1e-04) & 0.01 & 982 & 0.0293 & 983 & 0.0293 \\ \hline
\end{tabular}\label{Tab:velocity0}
\end{center}
}
\end{table}

In Table \ref{Tab:Damping}, we displayed the numerical results with different value of damping parameters $\eta$. With the appropriate choice of the damping parameter, say $\eta=2.5648 \times 10^{-3}$ in our example, the second order asymptotical regularization not only gives the most accurate result, but exhibits an acceleration affect. The critical value of the damping parameter, say $\eta=2/\Delta t$, also provides an accurate result. But it requires a few more steps. The influence of the damping parameter on the residual functional can be found in Figure \ref{Fig:damping}. It shows that at the same time point, the larger the damping parameter, the smaller the residual norm functional.

\begin{table}[tbhp]
{\footnotesize
\caption{The dependence of the solution accuracy and the convergence rate on the damping parameter $\eta$. $\Delta t=19.4946, x_0=1, \dot{x}_0=0, \tau=2, p=0.1125, \tau_{te}=1.1\times \delta^{4p/(4p+1)}$.}
\begin{center}
\begin{tabular}{|c|c|c|c|c|c|} \hline
\multirow{2}{*}{$\delta$ ($\delta'$)} &
\multirow{2}{*}{$\eta$} &
\multicolumn{2}{c|}{DP} &
\multicolumn{2}{c|}{TEDP} \\
\cline{3-6}
 &  & $k^*(\delta)$ & \textrm{L2Err}& $k^*(\delta)$ & \textrm{L2Err}  \\ \hline
8.7715e-04 (1e-03) & 2.5648e-05 & 6943 & 0.8319 & 17393 & 0.7744  \\
7.3673e-04 (1e-03) & 2.5648e-04 & 1108 & 0.1310 & 1728 & 0.1080  \\
7.1484e-04 (1e-03) & 2.5648e-03 & 124 & 0.0901 & 183 & 0.0870  \\
8.1011e-05 (1e-04) & 2.5648e-02 & 205 & 0.1101 & 206 & 0.1100  \\
7.3174e-05 (1e-04) & 0.0513 & 398 & 0.1108 & 398 & 0.1108  \\
7.2721e-05 (1e-04) & 0.1026 ($=2/\Delta t$) & 848 & 0.1096 & 848 & 0.1096 \\ \hline
\end{tabular}\label{Tab:Damping}
\end{center}
}
\end{table}

\subsection{Comparison with other methods}

In order to demonstrate the advantages of our algorithm over the traditional approaches, we solve the same problems by four famous iterative regularization methods~-- the Landweber method, the Nesterov's method, the $\nu$-method and the conjugate gradient method for the normal equation (CGNE, cf., e.g.~\cite{Hanke1995}). The Landweber method is given in (\ref{Linear}), while Nesterov's method is defined as~\cite{Nesterov-1983}
\begin{eqnarray}\label{Nesterov}
\left\{\begin{array}{ll}
z^k = x^k + \frac{k-1}{k+\alpha-1} (x^k - x^{k-1}),  \\
x^{k+1}=z^k + \Delta t A^*(y^\delta-A z^k),
\end{array}\right.
\end{eqnarray}
where $\alpha>3$ (we choose $\alpha=3.1$ in our simulations). Moreover, we select the Chebyshev method as our special $\nu$-method, i.e., $\nu=1/2$~\cite[\S~6.3]{engl1996regularization}. For all of four traditional iterative regularization methods, we use the Morozov's conventional discrepancy principle as the stopping rule.

We consider the following two different right-hand sides for the integral equation (\ref{IntegralEq}).
\begin{itemize}
\item \textbf{Example 1}: $y(s)=s(1-s)$. Then, $x^\dagger=2$, and $x^\dagger\in R((A^*A)^{p})$ for all $p<1/8$. This example uses the discretization size $n=400$. Other parameters are: $\Delta t=19.4946, \eta=2.5648\times 10^{-4}, x_0=1, \dot{x}_0=0, \tau=2, p=0.1125, \tau_{te}=1.1\times \delta^{4p/(4p+1)}$.
\item \textbf{Example 2}: $y(s)=s^4(1-s)^3$. Then, $x^\dagger=-6t^2(1-t)(2-8t+7t^2)$, and $x^\dagger\in R((A^*A)^{p})$ for all $p<5/8$. This example uses the discretization size $n=400$. Other parameters are: $\Delta t=19.4946, \eta=0.0051, x_0=0, \dot{x}_0=0, \tau=2, p=0.5625, \tau_{te}=1.1\times \delta^{4p/(4p+1)}$.
\end{itemize}

The results of the simulation are presented in Table \ref{tab:Comparison}, where we can conclude that, in general, the second order asymptotical regularization need less iteration and CPU time, and offers a more accurate regularization solution. Moreover, with respect to the Morozov's conventional discrepancy principle, the newly developed total energy discrepancy principle provides more accurate results in most cases. Concerning the number of iterations, the CGNE method performed much better than all of other methods. However, the accuracy of the CGNE method is much worse than other methods (including DP and TEDP), since the step size of CGNE is too large to capture the optimal point and the semi-convergence effect disturbs the iteration rather early.  Note that we set a maximal iteration number $k_{max}=400,000$ in all of our simulations.

\begin{table}[tbhp]
{\footnotesize
\caption{Comparisons with the Landweber method, the Nesterov's method, the Chebyshev method, and the CGNE method. }
\begin{center}
\begin{tabular}{|c|c|c|c|c|c|c|c|c|c|c|} \hline
\multirow{2}{*}{$\delta$} &
\multicolumn{3}{c|}{DP} &
\multicolumn{3}{c|}{TEDP} &
\multicolumn{3}{c|}{Landweber} \\
\cline{2-10}
& $k^*(\delta)$ & CPU (s) & \textrm{L2Err} & $k^*(\delta)$ & CPU (s) & \textrm{L2Err} & $k^*(\delta)$ & CPU (s) & \textrm{L2Err} \\ \hline
\multicolumn{10}{|c|}{\textbf{Example 1}}  \\  \hline
1.0400e-2 & 70 & 0.2023 & 0.1494 & 70 & 0.1995 & 0.1494 & 20438 & 52.6428 & 0.2639 \\
1.0380e-3 & 2872 & 8.9239 &  0.0945 & 2871 & 1.9105 & 0.0945  & $k_{max}$ & 1.6899e3& 0.1807 \\
1.0445e-4 & 56246 & 177.3387 & 0.0597 &  56246 & 177.9603 & 0.0473  & $k_{max}$ & 1.8048e3 & 0.1807 \\ \hline
\multicolumn{10}{|c|}{\textbf{Example 2}}  \\  \hline
1.0761e-2 & 15 & 0.0554 &  0.3676 & 25 & 0.0753 & 0.1987  & 1457 & 6.3672 & 0.7032 \\
1.0703e-3 & 93 & 0.2488 & 0.0637 & 132 & 0.3567 & 0.0581  & 23790 & 67.6082 & 0.1767 \\
1.1006e-4 & 453 & 1.4931 & 0.0195 & 531 & 1.8000 & 0.0184  & 187188 & 679.8274 & 0.0509 \\ \hline \hline
\multirow{2}{*}{$\delta$} &
\multicolumn{3}{c|}{Nesterov} &
\multicolumn{3}{c|}{Chebyshev} &
\multicolumn{3}{c|}{CGNE} \\
\cline{2-10}
& $k^*(\delta)$ & CPU (s) & \textrm{L2Err} & $k^*(\delta)$ & CPU (s) & \textrm{L2Err} & $k^*(\delta)$ & CPU (s) & \textrm{L2Err} \\ \hline
\multicolumn{10}{|c|}{\textbf{Example 1}}  \\  \hline
1.0400e-2 & 419 & 1.1384 & 0.2590 & 264 & 0.7378 & 0.2553  & 6 & 0.0351 & 0.2213 \\
1.0380e-3 & 2813 & 8.5986 &  0.1600 & 2229 & 7.6512 & 0.1496  & 18 & 0.0904& 0.1383  \\
1.0445e-4 & 16642 & 60.1179 & 0.1025 &  17443 & 52.2056 & 0.0897  & 39 & 0.2013& 0.0894 \\ \hline
\multicolumn{10}{|c|}{\textbf{Example 2}}  \\  \hline
1.0761e-2 & 102 & 0.3018 &  0.7043 & 62 & 0.1768 & 0.7102  & 6 & 0.0148 & 0.4835  \\
1.0703e-3 & 416 & 1.1732 & 0.1676 & 415 & 1.1908 & 0.1190  & 12 & 0.0261 & 0.1514 \\
1.1006e-4 & 1805 & 6.0793 & 0.0280 & 2226 & 7.6967 & 0.0196  & 15 & 0.0309 & 0.0447 \\ \hline
\end{tabular}\label{tab:Comparison}
\end{center}
}
\end{table}

\section{Conclusion and outlook}

In this paper, we have investigated the method of second order asymptotical regularization (SOAR) for solving the ill-posed linear inverse problem $Ax=y$ with the compact operator $A$ mapping between infinite dimensional Hilbert spaces. Instead of $y$, we are given noisy data $y^\delta$ obeying the inequality $\|y^\delta - y\|\leq \delta$. We have shown regularization properties for the dynamical solution of the second order equation (\ref{SecondFlow}). Moreover, by using  Morozov's conventional discrepancy principle on the one hand and a newly developed total energy discrepancy principle on the other hand, we have proven the order optimality of SOAR. Furthermore, based the framework of SOAR, by using the St\"{o}rmer-Verlet method, we have derived a novel iterative regularization algorithm. The convergence property of the proposed numerical algorithm is proven as well. Numerical experiments in Section 7 show that, in comparison with conventional iterative regularization methods, SOAR is a faster regularization method for solving linear inverse problems with high levels of accuracy.

Various numerical results show that the damping parameter $\eta$ in the second order equation (\ref{SecondFlow}) plays a prominent role in regularization and acceleration. Therefore, how to choose an optimal damping parameter should be studied in the future. Moreover, using the results of the nonlinear Landweber iteration, it will be possible to develop a theory of second order asymptotical regularization for wide classes of nonlinear ill-posed problems. Furthermore, it could be very interesting to investigate the case with the dynamical damping parameter $\eta=\eta(t)$. For instance, the second order equation (\ref{SecondFlow}) with $\eta=r/t$ ($r\geq3$) presents the continuous version of Nesterov's scheme~\cite{Su-2016}, and the discretization of (\ref{SecondFlow}) with
\begin{equation*}
\left\{\begin{array}{l}
\eta_k= \frac{(k+2\nu-1)(2k+4\nu-1)(2k+2\nu-3) - (k-1)(2k-3)(3k+3\nu-1)}{4(2k+2\nu-3)(2k+2\nu-1)(k+\nu-1)}, \\
\Delta t_k= 4 \frac{(2k+2\nu-1)(k+\nu-1)}{(k+2\nu-1)(2k+4\nu-1)},
\end{array}\right.
\end{equation*}
yields the well-known $\nu$-methods~\cite[\S~6.3]{engl1996regularization}. Even in the linear case (\ref{OperatorEq}), to the best of our knowledge, it is not quite clear whether Nesterov's approach equipped with a posteriori choice of the regularization parameter is an accelerated regularization method for solving ill-posed inverse problems. In our opinion, however, the second order asymptotical regularization can be a candidate for the systematic analysis of general second order regularization methods.

\section{Acknowledgement}

The work of Y. Zhang is supported by the Alexander von Humboldt foundation through a postdoctoral researcher fellowship, and the work of B. Hofmann is supported by the German Research Foundation (DFG-grant HO 1454/12-1).

\bibliographystyle{gAPA}
\bibliography{SecondOrderLinear}
\appendices

\section{Proofs in Section 3}

\subsection{Proof of Lemma~\ref{LemmaVanishing}}

Define the Lyapunov function of (\ref{SecondFlow}) by $\mathcal{E}(T) = \|A x(T)-y^\delta\|^2+ \|\dot{x}(T)\|^2$. It is not difficult to show that
\begin{equation}
\label{Decreasing}
\dot{\mathcal{E}}(t) = - 2 \eta \|\dot{x}(t)\|^2
\end{equation}
by looking at the equation (\ref{SecondFlow}) and the differentiation of the energy function
$\dot{\mathcal{E}}(t) = 2 \langle \dot{x}(t), \ddot{x}(t) - A^* ( y^\delta- A x(t) ) \rangle$. Hence, $\mathcal{E}(t)$ is non-increasing, and consequently, $\|\dot{x}(t)\|^2\leq \mathcal{E}(0)$. Therefore, $\dot{x}(\cdot)$ is uniform bounded.
Integrating both sides in (\ref{Decreasing}), we obtain
\begin{eqnarray*}
\int^{\infty}_{0} \|\dot{x}(t)\|^2 dt \leq  \mathcal{E}(0) / (2\eta) <  \infty,
\end{eqnarray*}
which yields $\dot{x}(\cdot) \in L^2([0,\infty),\mathcal{X})$.

Now, let us show that for any $x^*\in \mathcal{X}$ the following inequality holds.
\begin{eqnarray}\label{supLimit}
\mathop{\lim\sup}_{t\to \infty} \|A x(t)-y^\delta\| \leq \|A x^*-y^\delta\|.
\end{eqnarray}

Consider for every $t\in[0,\infty)$ the function $e(t)=e(t;x^*):=\frac{1}{2} \|x(t) - x^*\|$. Since $\dot{e}(t)= \langle x(t) - x^*, \dot{x}(t) \rangle $ and $\ddot{e}(t)= \|\dot{x}(t)\|^2 + \langle x(t) - x^*, \ddot{x}(t) \rangle$ for every $t\in[0,\infty)$. Taking into account (\ref{SecondFlow}), we get
\begin{equation}
\label{EqError2Sec3}
\ddot{e}(t) + \eta \dot{e}(t) + \langle x(t) - x^*, A^* ( A x(t) - y^\delta) \rangle = \|\dot{x}(t)\|^2.
\end{equation}

On the other hand, by the convexity inequality of the residual norm square functional $\|A x(t)-y^\delta\|^2$, we derive
\begin{eqnarray}\label{convexityIneqSec3}
\|A x(t)-y^\delta\|^2  + 2 \langle x^* - x(t), A^* ( A x(t) - y^\delta) \rangle \leq \|A x^*-y^\delta\|^2
\end{eqnarray}

Combine (\ref{EqError2Sec3}) and (\ref{convexityIneqSec3}) with the definition of $\mathcal{E}(t)$ to obtain
\begin{eqnarray*}
\ddot{e}(t) + \eta \dot{e}(t) \leq \frac{1}{2} \|A x^*-y^\delta\|^2 - \frac{1}{2} \mathcal{E}(t) + \frac{3}{2} \|\dot{x}(t)\|^2.
\end{eqnarray*}
By (\ref{Decreasing}), $\mathcal{E}(t)$ is nonincreasing, hence, given $t>0$, for all $\tau\in[0,t]$ we have
\begin{eqnarray*}
\label{ProofIneq}
\ddot{e}(\tau) + \eta \dot{e}(\tau) \leq \frac{1}{2} \|A x^*-y^\delta\|^2 - \frac{1}{2} \mathcal{E}(t) + \frac{3}{2} \|\dot{x}(\tau)\|^2.
\end{eqnarray*}
By multiplying this inequality with $e^{\eta t}$ and then integrating from 0 to $\theta$, we obtain
\begin{eqnarray*}
\dot{e}(\theta)\leq e^{-\eta \theta} \dot{e}(0) + \frac{1-e^{-\eta \theta}}{2\eta} (\|A x^*-y^\delta\|^2 - \mathcal{E}(t)) + \frac{3}{2} \int^\theta_0 e^{-\eta(\theta-\tau)} \|\dot{x}(\tau)\|^2 d \tau .
\end{eqnarray*}
Integrate the above inequality once more from 0 to $t$ together with the fact that $\mathcal{E}(t)$ decreases, to obtain
\begin{eqnarray}\label{ProofIneq2}
e(t)\leq e(0)+ \frac{1-e^{-\eta t}}{\eta} \dot{e}(0) + \frac{\eta t - 1 + e^{-\eta t}}{2\eta^2} (\|A x^*-y^\delta\|^2 - \mathcal{E}(t)) + h(t),
\end{eqnarray}
where $h(t):= \frac{3}{2} \int^t_0 \int^\theta_0 e^{-\eta(\theta-\tau)} \|\dot{x}(\tau)\|^2 d \tau d \theta$.

Since $e(t)\geq0$ and $\mathcal{E}(t)\geq \|A x(t)-y^\delta\|^2$, it follows from (\ref{ProofIneq2}) that
\begin{eqnarray*}
\frac{\eta t - 1 + e^{-\eta t}}{2\eta^2} \|A x(t)-y^\delta\|^2 \leq e(0)+ \frac{1-e^{-\eta t}}{\eta} \dot{e}(0) + \frac{\eta t - 1 + e^{-\eta t}}{2\eta^2} \|A x^*-y^\delta\|^2 + h(t).
\end{eqnarray*}
Dividing the above inequality by $\frac{\eta t - 1 + e^{-\eta t}}{2\eta^2}$ and letting $t\to\infty$, we deduce that
\begin{eqnarray*}
\mathop{\lim\sup}_{t\to \infty} \|A x(t)-y^\delta\|^2 \leq \|A x^*-y^\delta\|^2 + \mathop{\lim\sup}_{t\to \infty} \frac{2\eta}{t} h(t).
\end{eqnarray*}

Hence, for proving (\ref{supLimit}), it suffices to show that $h(\cdot)\in L^\infty([0,\infty),\mathcal{X})$. It is obviously held by noting the following inequalities
\begin{eqnarray*}
0 \leq h(t) = \frac{3}{2\eta} \int^t_0 (1- e^{-\eta(t-\tau)}) \|\dot{x}(\tau)\|^2 d \tau \leq \frac{3}{2\eta} \int^\infty_0 \|\dot{x}(\tau)\|^2 d \tau < \infty.
\end{eqnarray*}
From the inequality $\|A x(t)-y^\delta\|\geq \inf_{x^*\in \mathcal{X}} \|A x^*-y^\delta\|$, we conclude together with (\ref{supLimit}) that
\begin{equation}
\label{ProofLimit1}
\lim_{t\to \infty} \|A x(t)-y^\delta\|  = \inf_{x^*\in \mathcal{X}} \|A x^*-y^\delta\|.
\end{equation}
Consequently, we have
\begin{eqnarray*}
\lim_{t\to \infty} \|A x(t)-y^\delta\| \leq \|A x^\dagger-y^\delta\| \leq \delta.
\end{eqnarray*}

Now, let us show the remaining parts of the assertion. Since $\mathcal{E}(t)$ is nonincreasing and bounded from below by $\inf_{x^*\in \mathcal{X}} \|A x^*-y^\delta\|^2$, it converges as $t\to\infty$. If $\lim_{t\to\infty} \mathcal{E}(t) > \inf_{x^*\in \mathcal{X}} \|A x^*-y^\delta\|^2$, then $\lim_{t\to\infty} \|\dot{x}(t)\|>0$ by noting (\ref{ProofLimit1}). This contradicts the fact that $x(\cdot)\in L^2([0,\infty),\mathcal{X})$. Therefore, the limit (\ref{RateLimits}) holds and $\dot{x}(t)\to0$ as $t\to\infty$.

\subsection{Proof of Lemma~\ref{LemmaVanishingExact}}

(i) Consider for every $t\in[0,\infty)$ the function $e(t)=e(t;x^\dagger)=\frac{1}{2} \|x(t) - x^\dagger\|^2$. Similarly as in (\ref{EqError2Sec3}), it holds that
\begin{equation*}
\label{EqError2}
\ddot{e}(t) + \eta \dot{e}(t) + \langle x(t) - x^\dagger, A^* ( A x(t) - y) \rangle = \|\dot{x}(t)\|^2,
\end{equation*}
which implies that
\begin{eqnarray}
\label{ProofIneqe}
\ddot{e}(t) + \eta \dot{e}(t) + \frac{1}{\| A \|^2} \| A^* ( y- A x(t) ) \|^2 \leq \|\dot{x}(t)\|^2
\end{eqnarray}
or, equivalently (by using the evolution equation (\ref{SecondFlow})),
\begin{eqnarray}
\label{ProofIneq}
\ddot{e}(t) + \eta \dot{e}(t) + \frac{\eta}{\| A \|^2} \frac{d \|\dot{x}(t)\|^2}{dt} + \left( \frac{\eta^2}{\| A \|^2} - 1 \right) \|\dot{x}(t)\|^2 +  \frac{1}{\| A \|^2} \|\ddot{x}\|^2 \leq 0.
\end{eqnarray}

By the assumption $\eta\geq\|A\|$, we deduce that
\begin{equation}
\label{BoundedSolutionProof1}
\ddot{e}(t) + \eta \dot{e}(t) + \frac{\eta}{\| A \|^2} \frac{d \|\dot{x}(t)\|^2}{dt} \leq 0,
\end{equation}
which means that the function $t \mapsto \dot{e}(t) + \eta e(t)+ \frac{\eta}{\| A \|^2}  \|\dot{x}(t)\|^2$ is monotonically decreasing. Hence, a real number $C$ exists, such that
\begin{equation}
\label{ProofIneq5a}
\dot{e}(t) + \eta e(t)+ \frac{\eta}{\| A \|^2}  \|\dot{x}(t)\|^2 \leq C,
\end{equation}
which implies $\dot{e}(t) + \eta e(t) \leq C$. By multiplying this inequality with $e^{\eta t}$ and then integrating from 0 to $T$, we obtain the inequality
\begin{eqnarray*}
e(T)\leq e(0) e^{-\eta T} + C \left( 1 - e^{-\eta T} \right)/\eta \leq e(0) + C/\eta .
\end{eqnarray*}
Hence, $e(\cdot)$ is uniform bounded, and, consequently, the trajectory $x(\cdot)$ is uniform bounded.

(ii) follows from Lemma \ref{LemmaVanishing}.

Now, we prove assertion (iv). Define
\begin{eqnarray}\label{h}
h(t) = \frac{\eta}{2} \| x(t) - x^\dagger \|^2 + \langle \dot{x}(t), x(t) - x^\dagger \rangle.
\end{eqnarray}
By elementary calculations, we derive that
\begin{equation*}
\dot{h}(t) = \eta \langle \dot{x}(t), x(t) - x^\dagger \rangle + \langle \ddot{x}(t), x(t) - x^\dagger \rangle + \| \dot{x}(t) \|^2  = \| \dot{x}(t) \|^2 - \|Ax-y\|^2,
\end{equation*}
which implies that (by noting $\dot{\mathcal{E}}(t) = - 2 \eta \|\dot{x}(t) \|^2$)
\begin{eqnarray*}
\dot{\mathcal{E}}(t) + \eta \mathcal{E}(t)  + \eta \dot{h}(t) = 0.
\end{eqnarray*}

Integrate the above equation on $[0,T]$ to obtain together, with the nonnegativity of $\mathcal{E}(t)$,
\begin{eqnarray}\label{chiIneq}
\int^T_0 \mathcal{E}(t) dt = \frac{1}{\eta} \left( \mathcal{E}(0) - \mathcal{E}(t) \right) - (h(t)-h(0))
\leq \left(  \frac{1}{\eta} \mathcal{E}(0) + h(0) \right)- h(t).
\end{eqnarray}

On the other hand, since both $x(t)$ and $x^\dagger$ are uniform bounded, and $\dot{x}(t)\to0$ as $t\to0$, a constant $M$ exists such that $|h(t)|\leq M$. Hence, letting $T\to \infty$ in (\ref{chiIneq}), we obtain
\begin{eqnarray}\label{chiIneq2}
\int^\infty_0 \mathcal{E}(t)  dt < \infty.
\end{eqnarray}

Since $\mathcal{E}(t)$ is non-increasing, we deduce that
\begin{eqnarray}\label{chiIneq3}
\int^{T}_{T/2} \mathcal{E}(t) dt \geq \frac{T}{2}\mathcal{E}(T).
\end{eqnarray}
Using (\ref{chiIneq2}), the left side of (\ref{chiIneq3}) tends to 0 when $T\to\infty$, which implies that $\lim_{T\to\infty} T\mathcal{E}(T)=0$. Hence, we conclude $\lim_{T\to\infty} T \|A x(T) - y\|^2 =0$, which yields the desired result in (iv).

Finally, let us show the long-term behaviour of $\ddot{x}(\cdot)$. Integrating the inequality (\ref{ProofIneq}) from 0 to $T$ we obtain the fact that there exists a real number $C'$, such that for every $t\in[0,\infty)$
\begin{eqnarray}
\label{ProofIneqEddot}
\begin{array}{l}
\dot{e}(T) + \eta e(T) + \frac{\eta}{\| A \|^2} \|\dot{x}(T)\|^2 \\ \qquad\qquad
+ \left( \frac{\eta^2}{\| A \|^2} - 1 \right) \int^T_0 \|\dot{x}(t)\|^2 dt +  \frac{1}{\| A \|^2} \int^T_0 \|\ddot{x}(t)\|^2 dT \leq C'.
\end{array}
\end{eqnarray}
Since both $e(\cdot)$ and $\dot{e}(\cdot)$ are global bounded (note that $x(\cdot),\dot{x}(\cdot)\in L^\infty([0,\infty),\mathcal{X})$), inequality (\ref{ProofIneqEddot}) gives $\ddot{x}(\cdot)\in L^2([0,\infty),\mathcal{X})$. The relations $\ddot{x}(\cdot)\in L^\infty([0,\infty),\mathcal{X})$ and $\ddot{x}(t)\to0$ as $t\to\infty$ are obvious by noting assertions (i), (ii), (iv) and the connection equation (\ref{SecondFlow}).


\section{Convergence analysis of the second order asymptotical regularization for the case when $\eta\in (0, 2\|A\|]$}

\subsection{The underdamped case: $2\sigma_{j_0+1} < \eta< 2\sigma_{j_0}$}

In this case, the solution to the second order differential equation (\ref{SecondFlow}) reads
\begin{eqnarray*}\label{xDeltaReguCase2}
\begin{array}{ll}
x(t) = (1- A^* A g^{ab}(t, A^* A)) x_0 + \phi^{ab}(t, A^* A)\dot{x}_0 + g^{ab}(t, A^* A) A^* y^\delta,
\end{array}
\end{eqnarray*}
where
\begin{eqnarray*}\label{gtLambdaCase2}
& g^{ab}(t, \lambda) =
\left\{\begin{array}{ll}
g(t, \lambda), ~ \textrm{~if~}  \lambda < \eta^2/4, \\
g^{b}(t, \lambda), ~ \textrm{~if~}  \lambda > \eta^2/4,
\end{array}\right.~ \phi^{ab}(t, \lambda) =
\left\{\begin{array}{ll}
\phi(t, \lambda), ~ \textrm{~if~}  \lambda < \eta^2/4, \\
\phi^{b}(t, \lambda), ~ \textrm{~if~}  \lambda > \eta^2/4,
\end{array}\right.
\end{eqnarray*}
where $g(t, \lambda)$ and $\phi(t, \lambda)$ are defined in (\ref{gPhiDef}), and
\begin{eqnarray}\label{gPhiDef2}
\left\{\begin{array}{ll}
g^{b}(t, \lambda) = \frac{1}{\lambda} \left\{ 1- e^{-\frac{\eta}{2} t} \left[ \frac{\eta}{\sqrt{4 \lambda - \eta^2}} \sin \left( \frac{\sqrt{4 \lambda - \eta^2}}{2} t \right)+ \cos \left( \frac{\sqrt{4 \lambda - \eta^2}}{2} t \right) \right] \right\} , \\
\phi^{b}(t, \lambda)= \frac{2}{\sqrt{4 \lambda - \eta^2}} e^{-\frac{\eta}{2} t} \sin \left( \frac{\sqrt{4 \lambda - \eta^2}}{2} t \right).
\end{array}\right.
\end{eqnarray}

As in the overdamped case, we define
\begin{equation}\label{gAlphaFunCase2}
g^{ab}_\alpha(\lambda) = g^{ab}(1/\alpha, \lambda) \quad \textrm{~and~} \quad \phi^{ab}_\alpha(\lambda) = \phi^{ab}(1/\alpha, \lambda).
\end{equation}
In this case, the corresponding bias function becomes
\begin{eqnarray*}\label{rAlphaCase2}
& r^{ab}_\alpha(\lambda) =
\left\{\begin{array}{ll}
r_\alpha(\lambda), \qquad \textrm{~if~}  \lambda < \eta^2/4, \\
r^{b}_\alpha(\lambda) := e^{-\frac{\eta}{2\alpha}} \left[ \frac{\eta}{\sqrt{4 \lambda - \eta^2}} \sin \left( \frac{\sqrt{4 \lambda - \eta^2}}{2\alpha} \right)+ \cos \left( \frac{\sqrt{4 \lambda - \eta^2}}{2\alpha} \right) \right],~ \textrm{~if~}  \lambda > \eta^2/4,
\end{array}\right.
\end{eqnarray*}
where $r_\alpha(\lambda)$ is given in (\ref{rAlpha}).

\begin{theorem}\label{ThmReguCase2}
The functions $\{g^{ab}_\alpha(\lambda), \phi^{ab}_\alpha(\lambda)\}$, defined in (\ref{gAlphaFunCase2}), satisfy the conditions $(i) - (iii)$ of Proposition~\ref{gAlpha}.
\end{theorem}

\begin{proof}
The first requirement in Proposition \ref{gAlpha} is obvious. Furthermore, using the inequalities $|\sin \xi| \leq |\xi|$ and $e^{-\xi} \left( \xi + 1 \right) \leq 1$ we obtain
\begin{eqnarray}\label{Case2ProofIneq1}
\left| e^{-\frac{\eta}{2\alpha}} \left[ \frac{\eta}{\sqrt{4 \lambda - \eta^2}} \sin \left( \frac{\sqrt{4 \lambda - \eta^2}}{2\alpha} \right)+ \cos \left( \frac{\sqrt{4 \lambda - \eta^2}}{2\alpha} \right) \right] \right| \leq 1,
\end{eqnarray}
which implies the second condition in Proposition \ref{gAlpha}: $|r^{ab}_{\alpha}(\lambda)|\leq \gamma^{ab}_1$ with $\gamma^{ab}_1:=\max\left\{ \frac{\eta}{2\sqrt{\eta^2 - 4 \sigma^2_{j_0+1}}} + \frac{1}{2}, 1 \right\}$. Similarly, we have
\begin{eqnarray}\label{Case2ProofIneq1}
\left| \phi^{b}_\alpha (\lambda) \right|= \left| \frac{2}{\sqrt{4 \lambda - \eta^2}} e^{-\frac{\eta}{2\alpha}} \sin \left( \frac{\sqrt{4 \lambda - \eta^2}}{2\alpha} \right) \right| \leq \alpha^{-1} e^{-\frac{\eta}{2\alpha}} \leq \frac{2}{e\eta},
\end{eqnarray}
which means that $|\phi^{ab}_{\alpha}(\lambda)|\leq \gamma^{ab}_2$ with
$\gamma^{ab}_2:=\max\left\{ \frac{\eta}{2\sqrt{\eta^2 - 4 \sigma^2_{j_0+1}}}, \frac{2}{e\eta} \right\}$.

Now, let us check the third condition in Proposition \ref{gAlpha}. Using the inequality (\ref{Case2ProofIneq1}) we obtain that for $\lambda > \eta^2/4$
\begin{eqnarray*}
\begin{array}{ll}
\frac{1}{\sqrt{\lambda}} \left\{ 1- e^{-\frac{\eta}{2\alpha}} \left[ \frac{\eta}{\sqrt{4 \lambda - \eta^2}} \sin \left( \frac{\sqrt{4 \lambda - \eta^2}}{2\alpha} \right)+ \cos \left( \frac{\sqrt{4 \lambda - \eta^2}}{2\alpha} \right) \right] \right\} \leq \frac{2}{\sqrt{\lambda}} \leq \frac{4}{\eta}.
\end{array}
\end{eqnarray*}
Hence, in the case when $\lambda > \eta^2/4$ and $\alpha\leq \eta^2$, we have $\sqrt{\lambda} |g^b_{\alpha}(\lambda)| \leq  4/ \eta  \leq 4 / \sqrt{\alpha}$.

Finally, by defining
\begin{eqnarray}\label{gamma_ab}
\gamma^{ab}_* = \max \left\{ \sqrt{\eta/(\eta^2 - 4 \sigma^2_{j_0+1})} , 4 \right\}
\end{eqnarray}
we can deduce that $\sqrt{\lambda} |g^{ab}_{\alpha}(\lambda)|\leq \gamma^{ab}_*/ \sqrt{\alpha}$ for $\alpha \in (0,\bar{\alpha}]$ with $\bar{\alpha}=\eta^2$.

\end{proof}

\begin{proposition}\label{ThmQualificationCase2}
For all exponents $p>0$ the monomials $\psi(\lambda) = \lambda^{p}$ are qualifications with the constants $\gamma_{ab} = \max \{\gamma, \gamma_b\}$ for the second order asymptotical regularization method in the underdamped case, where $\gamma$ is defined in (\ref{gamma_a}) and
\begin{eqnarray}\label{gamma_b}
\gamma_b:= \frac{\eta+2\|A\|^2}{2} \left( \frac{2(p+1)}{e\eta} \right)^{p+1} \|A\|^{2p}.
\end{eqnarray}
\end{proposition}

\begin{proof}
By Proposition \ref{ThmQualificationCase1}, we only need to show that
\begin{eqnarray}\label{QualificationFunProofCase2}
\sup_{\lambda\in(\eta^2/4,\|A\|^2]} |r^b_{\alpha}(\lambda)| \lambda^{p} \leq \gamma_b \alpha^{p} \quad \textrm{~and~} \sup_{\lambda\in(\eta^2/4,\|A\|^2]} |\phi^b_{\alpha}(\lambda)| \lambda^{p} \leq \gamma_b \alpha^{p}
\end{eqnarray}
for all $\alpha\in(0, \|A\|^2]$. Set $\xi=\eta/2$ and $p=p'+1$ in (\ref{Pflimiting1}) to obtain
\begin{eqnarray}\label{Case2ProofIneq2}
e^{-\frac{\eta}{2\alpha}} \leq  \left( 2(p'+1) / (e\eta) \right)^{p'+1} \alpha^{p'+1}.
\end{eqnarray}
Then, using (\ref{Case2ProofIneq1}) and (\ref{Case2ProofIneq2}), we can derive that for all $\alpha\in(0, \|A\|^2]$
\begin{eqnarray*}\label{QualificationFunProofCase2a}
\begin{array}{ll}
\left| e^{-\frac{\eta}{2\alpha}} \left[ \frac{\eta}{\sqrt{4 \lambda - \eta^2}} \sin \left( \frac{\sqrt{4 \lambda - \eta^2}}{2\alpha} \right)+ \cos \left( \frac{\sqrt{4 \lambda - \eta^2}}{2\alpha} \right) \right] \right| \cdot \lambda^{p'}
\leq e^{-\frac{\eta}{2\alpha}} \left( \frac{\eta}{2\alpha} + 1 \right) \lambda^{p'} \\
\leq e^{-\frac{\eta}{2\alpha}} \cdot \frac{\eta+2\|A\|^2}{2\alpha} \cdot \|A\|^{2p'}
\leq \left( \frac{2(p'+1)}{e\eta} \right)^{p'+1} \cdot \alpha^{p'+1} \cdot \frac{\eta+2\|A\|^2}{2\alpha} \cdot \|A\|^{2p'} = \gamma_b \alpha^{p'},
\end{array}
\end{eqnarray*}
which yields the first inequality in (\ref{QualificationFunProofCase2}).

Finally, from the above result, we can deduce that for all $\alpha\in(0, \|A\|^2]$
\begin{eqnarray*}\label{QualificationFunProofCase2b}
\begin{array}{ll}
|\phi^b_{\alpha}(\lambda)| \lambda^{p'} \leq e^{-\frac{\eta}{2\alpha}} \frac{\eta}{2\alpha} \lambda^{p'} \leq e^{-\frac{\eta}{2\alpha}} \left( \frac{\eta}{2\alpha} + 1 \right) \lambda^{p'} \leq  \gamma_b \alpha^{p'},
\end{array}
\end{eqnarray*}
which completes the proof.
\end{proof}

\subsection{The critical damping case: $\eta= 2\sigma_{j_0}$}

In this case, the solution of (\ref{SecondFlow}) is $x(t) = (1- A^* A g^{abc}(t, A^* A)) x_0 + \phi^{abc}(t, A^* A)\dot{x}_0 + g^{abc}(t, A^* A) A^* y^\delta$, where
\begin{eqnarray*}\label{gtLambdaCase2}
& g^{abc}(t, \lambda) =
\left\{\begin{array}{ll}
g(t, \lambda), ~ \textrm{~if~}  \lambda > \eta^2/4, \\
g^{b}(t, \lambda), ~ \textrm{~if~}  \lambda < \eta^2/4,  \\
g^{c}(t, \lambda), ~  \lambda = \eta^2/4,
\end{array}\right. ~ \phi^{abc}(t, \lambda) =
\left\{\begin{array}{ll}
\phi(t, \lambda), ~ \textrm{~if~}  \lambda > \eta^2/4, \\
\phi^{b}(t, \lambda), ~ \textrm{~if~}  \lambda < \eta^2/4,  \\
\phi^{c}(t, \lambda), ~  \lambda = \eta^2/4,
\end{array}\right.
\end{eqnarray*}
and $g^{c}(t, \lambda):= \frac{4}{\eta^2} \left\{ 1- e^{-\frac{\eta}{2} t} \left( \frac{\eta}{2} t +1 \right) \right\}$, $\phi^{c}(t, \lambda):=  t e^{-\frac{\eta}{2} t}$.

Define
\begin{equation}\label{gAlphaFunCase3}
g^{abc}_\alpha(\lambda) = g^{abc}(1/\alpha, \lambda) \quad \textrm{~and~} \quad \phi^{abc}_\alpha(\lambda) = \phi^{abc}(1/\alpha, \lambda),
\end{equation}
and obtain the corresponding bias function
\begin{eqnarray*}\label{rAlphaCase3}
& r^{abc}_\alpha(\lambda) =
\left\{\begin{array}{ll}
r^{a}_\alpha(\lambda), \qquad  \textrm{~if~}  \lambda > \eta^2/4, \\
r^{b}_\alpha(\lambda), \qquad \textrm{~if~}  \lambda < \eta^2/4,  \\
r^{c}_\alpha(\lambda) := e^{-\frac{\eta}{2\alpha}} \left( \frac{\eta}{2\alpha} +1 \right), \quad \textrm{~if~}  \lambda = \eta^2/4.
\end{array}\right.
\end{eqnarray*}

\begin{theorem}\label{ThmReguCase3}
The functions $\{ g^{abc}_\alpha(\lambda) , \phi^{abc}_\alpha(\lambda) \}$, given in (\ref{gAlphaFunCase3}), satisfy the conditions $(i) - (iii)$ of Proposition~\ref{gAlpha}.
\end{theorem}

\begin{proof}
By Theorem \ref{ThmReguCase2}, we only need to check the case when $\lambda = \eta^2/4$. In this case, it is easy to verify that $\lim_{\alpha\to0} \phi_{\alpha}(\lambda) =\lim_{\alpha\to0} e^{-\frac{\eta}{2\alpha}} / \alpha =0$, $\lim_{\alpha\to0} r_{\alpha}(\lambda) =\lim_{\alpha\to0} e^{-\frac{\eta}{2\alpha}} \left( \frac{\eta}{2\alpha} +1 \right) =0$ and $|\phi_{\alpha}(\lambda)|\leq \frac{2}{e\eta}$ and $|r_{\alpha}(\lambda)|\leq 1$ for all $\alpha> 0$. Finally, by the inequality (assume that $\alpha\leq \eta^2$) $\frac{1}{\sqrt{\lambda}} \left\{ 1- e^{-\frac{\eta}{2\alpha}} \left( \frac{\eta}{2\alpha} +1 \right) \right\} \leq \frac{2}{\sqrt{\lambda}} = \frac{4}{\eta} \leq  \frac{4}{\sqrt{\alpha}}$, and Theorem \ref{ThmReguCase2}, we complete the proof with $\sqrt{\lambda} |g_{\alpha}(\lambda)|\leq \gamma^{ab}_*/ \sqrt{\alpha}$ for $\alpha \in (0,\bar{\alpha}]$, $\bar{\alpha}=\eta^2$, and $\gamma^{ab}_*$ is defined in (\ref{gamma_ab}).
\end{proof}

\begin{proposition}\label{ThmQualificationCase3}
For all exponents $p>0$ the monomials $\psi(\lambda) = \lambda^{p}$ are qualifications with the constants $\gamma_{abc} = \max \left\{ \gamma, \gamma_b, \gamma_c \right\}$ for the second order asymptotical regularization method in the critical damping case, where
\begin{eqnarray}\label{gamma_c}
\gamma_c:= \frac{\eta+2\|A\|^2}{2} \left( \frac{p+1}{e} \right)^{p+1} \left( \frac{\eta}{2} \right)^{p-2} \max\left(\frac{\eta}{2} , 1 \right),
\end{eqnarray}
and the constants $\gamma$ and $\gamma_b$ are defined in (\ref{gamma_a}) and (\ref{gamma_b}) respectively.
\end{proposition}

\begin{proof}
By Propositions \ref{ThmQualificationCase1} and \ref{ThmQualificationCase2}, we only need to show that for all $\alpha\in(0, \|A\|^2]$
\begin{eqnarray*}
|r_{\alpha}(\eta^2/4)| (\eta^2/4)^p \leq \gamma_c \alpha^{p} \textrm{~and~} |\phi_{\alpha}(\eta^2/4)| (\eta^2/4)^p \leq \gamma_c \alpha^{p}.
\end{eqnarray*}

By (\ref{QualificationFunProofCase2a}) and elementary calculations, we derive that
\begin{eqnarray*}
\begin{array}{ll}
\left|r_{\alpha} \left( \frac{\eta^2}{4} \right) \right| \left( \frac{\eta^{2}}{4} \right)^p = e^{-\frac{\eta}{2\alpha}} \left( \frac{\eta}{2\alpha} +1 \right) \left( \frac{\eta^{2}}{4} \right)^p
\leq \left( \left( \frac{2(p+1)}{e\eta} \right)^{p+1} \alpha^{p+1} \right) \frac{\eta+2\|A\|^2}{2\alpha}   \left( \frac{\eta^{2}}{4} \right)^p \\ \quad
= \left\{ \frac{\eta+2\|A\|^2}{2} \left( \frac{2(p+1)}{e\eta} \right)^{p+1} \left( \frac{\eta^{2}}{4} \right)^p \right\} \alpha^{p} = \left\{ \frac{\eta+2\|A\|^2}{2} \left( \frac{p+1}{e} \right)^{p+1} \left( \frac{\eta}{2} \right)^{p-1} \right\} \alpha^{p}  \leq \gamma_c \alpha^{p},
\end{array}
\end{eqnarray*}
and
\begin{eqnarray*}
\begin{array}{ll}
\left|\phi_{\alpha} \left( \frac{\eta^2}{4} \right) \right| \left( \frac{\eta^{2}}{4} \right)^p = e^{-\frac{\eta}{2\alpha}} \frac{\eta}{2\alpha} \frac{2}{\eta} \left( \frac{\eta^{2}}{4} \right)^p
\leq e^{-\frac{\eta}{2\alpha}} \left( \frac{\eta}{2\alpha} +1 \right) \left( \frac{\eta}{2} \right)^{2p-1}
 \\ \qquad
\leq \left( \left( \frac{p+1}{e} \right)^{p+1} \alpha^{p+1} \right) \frac{\eta+2\|A\|^2}{2\alpha}  \left( \frac{\eta}{2} \right)^{2p-1}  \leq \gamma_c \alpha^{p},
\end{array}
\end{eqnarray*}
which yields the required result.
\end{proof}

\end{document}